\documentclass[journal]{IEEEtran}

\usepackage{cite}
\usepackage{amsmath,amssymb,amsfonts,amsthm}
\usepackage{algorithm}
\usepackage{algorithmic}
\usepackage{epsfig} 
\usepackage{comment}
\usepackage[algo2e,linesnumbered,ruled]{algorithm2e}
\usepackage{textcomp}
\def\BibTeX{{\rm B\kern-.05em{\sc i\kern-.025em b}\kern-.08em
    T\kern-.1667em\lower.7ex\hbox{E}\kern-.125emX}}
\usepackage{float}
\PassOptionsToPackage{hyphens}{url}\usepackage{hyperref}

\makeatletter
\newcommand{\removelatexerror}{\let\@latex@error\@gobble}
\makeatother
\usepackage{color}

\usepackage[utf8]{inputenc}
\usepackage{mathtools}  
\usepackage{nccmath}

\usepackage{dsfont}
\theoremstyle{plain}
\newtheorem{theorem}{Theorem}
\newtheorem{definition}{Definition}
\newtheorem{assumption}{Assumption}

\newtheorem{remark}[theorem]{Remark}
\newtheorem{lemma}{Lemma}

\newtheorem{property}[theorem]{Property}

\newcommand{\mR}{{\mathbb R}}
\newcommand{\mE}{{\mathbb E}}

\newcommand{\mP}{{\mathbb P}}

\newcommand{\mU}{{\mathbb U}}

\newcommand{\bk}{{\mathbf k}}

\newcommand{\bq}{{\mathbf q}}
\newcommand{\bP}{{\mathbf P}}

\newcommand{\bM}{{\mathbf M}}

\newcommand{\bx}{{\mathbf x}}
\newcommand{\brho}{{\boldsymbol \rho}}

\newcommand{\bPsi}{{\boldsymbol \Psi}}
\newcommand{\bPhi}{{\boldsymbol \Phi}}

\newcommand{\bb}{{\mathbf b}}
\newcommand{\bX}{{\mathbf X}}
\newcommand{\bY}{{\mathbf Y}}
\newcommand{\bg}{{\mathbf g}}
\newcommand{\bA}{{\mathbf A}}
\newcommand{\bB}{{\mathbf B}}
\newcommand{\bu}{{\mathbf u}}
\newcommand{\bL}{{\mathbf L}}

\newcommand{\bW}{{\mathbf W}}

\newcommand{\bK}{{\mathbf K}}

\newcommand{\bG}{{\mathbf G}}

\newcommand{\bS}{{\mathbf S}}

\newcommand{\bv}{{\mathbf v}}

\newcommand{\bn}{{\mathbf n}}
\newcommand{\by}{{\mathbf y}}

\newtheorem*{proof}{Proof}

\definecolor{mycolor1}{RGB}{230,97,1}
\definecolor{mycolor2}{RGB}{153,204,0}
\definecolor{mycolor3}{RGB}{253,184,99}
\definecolor{mycolor4}{RGB}{153,204,255}
\begin{document}

\title{Data-Driven Stochastic Optimal Control using Linear Transfer Operators}
\author{Umesh Vaidya and Duvan Tellez-Castro
\thanks{
}
\thanks{ 
This work was supported by the Automotive Research Center (ARC), a US Army Center of Excellence for modeling and simulation of ground vehicles, under Cooperative Agreement W56HZV-19-2-0001 with the US Army DEVCOM Ground Vehicle Systems Center (GVSC).DISTRIBUTION A. Approved for public release; distribution unlimited. OPSEC\#5906.  Umesh Vaidya, is Professor at the Department of Mechanical Engineering, Clemson University, Clemson, SC 29631 USA (e-mail: uvaidya@clemson.edu). 
}}
\maketitle


\begin{abstract}
We provide a data-driven framework for optimal control of a continuous-time stochastic dynamical system. The proposed framework relies on the linear operator theory involving linear Perron-Frobenius (P-F) and Koopman operators. Our first results involving the P-F operator provide a convex formulation to the optimal control problem in the dual space of densities.
This convex formulation of the stochastic optimal control problem leads to an infinite-dimensional convex program. The finite-dimensional approximation of the convex program is obtained using a data-driven approximation of the P-F operator. Our second results demonstrate the use of the Koopman operator, which is dual to the P-F operator, for the stochastic optimal control design. We show that the Hamilton Jacobi Bellman (HJB) equation can be expressed using the Koopman operator. We provide an iterative procedure along the lines of a popular policy iteration algorithm based on the data-driven approximation of the Koopman operator for solving the HJB equation. The two formulations, namely the convex formulation involving P-F operator and Koopman based formulation using HJB equation, can be viewed as dual to each other where the duality follows due to the dual nature of P-F and Koopman operators. Finally, we present several numerical examples to demonstrate the efficacy of the developed framework. 

\end{abstract}


\section{Introduction}

The stochastic optimal control problem (SOCP) is a cornerstone of systems and control theory \cite{aastrom2012introduction,kushner1964dynamical}. This problem has received renewed attention with the growing interest in data-driven analytics and control with applications ranging from vehicle autonomy, robotics, transportation networks, power grid, security, and advanced manufacturing \cite{kormushev2013reinforcement,sallab2017deep, sethi2002optimal}. The SOCP is also at the heart of Reinforcement learning (RL), and a variety of algorithms are developed for the data-driven approximation of its solution \cite{sutton2018reinforcement}. For a system in continuous-time with continuous state space and control, the solution to SOCP essentially boils down to solving a Hamilton Jacobi Bellman (HJB) equation \cite{fleming2012deterministic}, which is a nonlinear partial differential equation. In discrete-time, SOCP involves solving the Bellman equation using the principle of dynamic programming \cite{bertsekas1996stochastic}. Thus, the Bellman equation can be viewed as the discrete-time counterpart of the continuous-time HJB equation. Given the nonlinear nature of the HJB equation, one of the popular approaches to solve the HJB equation is via iterative approach \cite{beard1997galerkin,bertsekas2011approximate}. 
In this paper, we propose an alternate approach for solving SOCP based on the convex formulation of the problem in the dual space of densities. This paper provides a data-driven solution to the SOCP over an infinite time horizon with continuous-time system dynamics.  The dual approach leads to a convex infinite-dimensional optimization problem to be solved for the SOCP. Unlike iterative algorithms for solving HJB equation in the primal domain, the convex formulation in the dual space lends itself to a single-shot approach for solving SOCP. We use a linear operator theoretic framework involving P-F and Koopman operators \cite{Lasota} to provide a novel perspective to the SOCP problem and the computation of its solution using data. We show that the traditional primal formulation of SOCP involving the HJB equation is closely tied to the Koopman operator. Furthermore, the dual convex formulation of the SOCP can be understood naturally through the lenses of 
duality between the Koopman and P-F operator.
\\

\noindent {\it Literature review}: Given the significance of SOCP in various applications, there is extensive literature on this topic. We refer the interested readers to the survey articles and classical work on this topic \cite{pham2005some,kushner2001numerical,kushner1964dynamical,bertsekas1996stochastic,bertsekas4}.  With the nonlinear and infinite-dimensional nature of the HJB equation, an analytical solution can be found in very few cases, and one has to resort to numerical methods for solving the HJB equation. In the development of numerical methods, the complexity associated with the nonlinear nature of the HJB is broken down by providing an iterative process for solving the HJB equation. The iterative approach relies on solving an infinite-dimensional linear equation for the value function with a given control input. The value function is then used to update the control input.  The infinite-dimensional linear equation for value function is solved approximately for the value function using Galerkin-type projection scheme \cite{beard1997galerkin}. The iterative approach for solving SOCP via HJB equation and also Bellman equation plays a fundamental role in the variety of RL algorithms, including policy iteration, value iteration, and actor-critic method \cite{kumar2009computational}. 
The dual approach to the SOCP is not studied extensively from the numerical perspective. 
In particular, it is known that the dual formulation to the Bellman equation leads to an infinite-dimensional linear program to the design of SOC \cite{de2003linear}. However, contrary to the dual formulation proposed in this paper, the Bellman dual infinite-dimensional linear program is constructed on the joint space of states and control. With continuous state and action space, the finite-dimensional approximation of the bi-infinite linear program is a challenging problem. On the other hand, the dual formulation to the OCP is well studied for deterministic control systems \cite{lasserre2008nonlinear}. The results in \cite{korda2017convergence} use sum-of-square (SOS)-based computational methods and moment-based relaxation techniques for the finite-dimensional approximation of infinite-dimensional convex optimization problem in the dual form. The moment-based relaxation and SOS-based optimization methods are also used for the analysis of the stochastic dynamical system, including solving optimal control problems in its dual form and computing exit time \cite{savorgnan2009discrete, henrion2021moment}. The results developed in this paper can be viewed as a natural extension of the results developed in \cite{vaidyaocpconvex} from the deterministic to the stochastic setting, where computational methods based on the finite-dimensional approximation of linear P-F and Koopman operators are developed for data-driven control. More recently there has been explosion of research activities on  data-driven control design using linear Koopman operator \cite{brunton2016koopman,peitz2019koopman,villanueva2021towards,borggaard2009control, abraham2019active,korda2018linear, sootla2018optimal,otto2021koopman,fackeldey2020approximative,kaiser2021data}. However, most of these results only exploit the linearity of the Koopman operator and which does not lead to convex formulation to controller synthesis problem. Another popular approach for solving the HJB equation relies on path integral-based numerical scheme \cite{satoh2016iterative}. The basis idea is to perform a change of variables that use noise statistics of the underlying stochastic system to transform the nonlinear HJB equation to a linear partial differential equation (PDE). The solution to the linear PDE is then obtained using the Feynman-Kac formula and path integral.  \\

\noindent {\it Contributions}: We provide a convex formulation to the SOCP using the linear operator theoretic framework involving  P-F and Koopman operators. The linear P-F and Koopman operators are dual to each other and provide for a linear lifting of nonlinear system dynamics in the space of density and function (observables), respectively. The results are inspired by the dynamical system theory, as the duality in SOCP is discovered through duality between the P-F and Koopman operators. 
The SOCP problem is formulated in the dual space of density using P-F operator-based lifting of control system dynamics. This dual approach leads to the infinite-dimensional convex optimization-based formulation of the SOCP.  We provide a computational framework based on the data-driven approximation of the P-F operator for the data-driven stochastic optimal control design. 
The convex formulation of SOCP is made possible by exploiting the P-F operator's linearity, positivity, and Markov properties. Furthermore, we show that the hard constraints on the control input and the state can also be written convexly in the dual formulation. 
On the other hand, we establish a connection between the Koopman operator and the HJB equation. This connection allows us to develop a numerical algorithm for the data-driven solution of the HJB equation based on Koopman theory. In particular, we provide an iterative algorithm based on a data-driven approximation of the Koopman operator for solving the SOCP problem in the primal domain. This new algorithm is reminiscent of the generalized policy iteration (GPI) algorithm in RL and we call it as Koopman policy iteration (KPI). Moreover, the interpretation of GPI using the Koopman theory opens up the possibility of exploiting the rich spectral theory of the Koopman operator for data-driven control.  
It is important to emphasize that the existing iterative algorithm for solving the HJB equation, including our proposed Koopman-based approach, requires an initial control policy to be stabilizing. However, designing stabilizing controller for a stochastic nonlinear system is far from a trivial problem. Our proposed dual approach to SOCP does not suffer from this drawback. The convex optimization problem in our dual framework can be solved as a single shot problem, where almost everywhere stochastic stabilizability arises as constraints of this optimization problem. So the data-driven stochastic stabilization will emerge as the particular case of the main result on SOCP. The results in this paper are an extended version of results from \cite{vaidya2021convex}. In particular, the data-driven computation framework is new to this paper. Results involving convex formulation to SOCP with state and input constraints are also new. 
Furthermore, we provide precise characterization for the existence of optimal controller along with the proof of some of the key results which were missing from \cite{vaidya2021convex}. 
The results involving the Koopman-based formulation of the HJB equation and the associated iterative computation scheme based on the data-driven approximation of the Koopman operator are also new to this paper. \\

\noindent {\it Organization}: The paper is organized as follows. In Section \ref{section_prelim}, we provide preliminaries on linear transfer operator theory involving P-F and Koopman operators their semi-group and infinitesimal generators. In Section \ref{section_stochasticstability}, we prove new results on the stochastic stability analysis. These results play an important role in the dual formulation of the SOCP. The main results of this paper on the dual formulation of SOCP involving P-F operator and primal formulation involving Koopman operator are presented in Section  \ref{section_main}.  The computation framework for the data-driven approximation of stochastic optimal control is presented in Section \ref{section_data-driven}. Conclusions are presented in Section \ref{section_conclusion}.

\section{Preliminaries and Notations}\label{section_prelim}

In this section, we discuss some preliminaries and introduce some notations, which are used in deriving the main results on data-driven optimal control. Consider a stochastic dynamical system 
\begin{eqnarray}
\dot \bx={\bf F}(\bx)+\sigma\bn(\bx)\xi,\;\;\;\label{sys}
\end{eqnarray}
where, $\bx\in \mathbb{R}^n,  \mathbb{R}^q\ni\xi=\frac{dw}{dt}$ is the white noise process and is time derivative of Weineer process. We assume that ${\bf f}(0)=\bn(0)=0$ and hence the origin is assumed to be the equilibrium point of the system.  Let $\bX_t$ be the solution of the stochastic differential equation (\ref{sys}). For more details on the definition and condition for the existence of solution of stochastic differential equation, refer to \cite{Lasota}. In particular, the vector fields ${\bf F}(\bx)$ and $\bn(\bx)$ are assumed to satisfy following Lipschitz condition.
\begin{eqnarray}
|{\bf F}(\bx)-{\bf F}(\by)|\leq L|\bx-\by|,\;\;\;\bx\in \mathbb{R}^n,\;\;\by\in \mathbb{R}^n\nonumber\\
|{\bf n}(\bx)-{\bf n}(\by)|\leq L|\bx-\by|,\;\;\bx\in \mathbb{R}^n,\;\;\by\in \mathbb{R}^n
\end{eqnarray}
for some constant $L$. 

\noindent {\bf Notations}: Let $\bX^\bx:=\{\bX_t^\bx\}_{t\geq 0}$ be the solution process with initial condition $\bX_0=\bx$, where the solution is defined in the sense of Ito calculus \cite{Lasota}. Let $P^\bx$ be the distribution of $\bX^\bx$ and $\mE_\bx$ be the expectation with respect to $P^\bx$. We introduce following notations. Let ${\cal L}_1(\mR^n,\mR^m)$, integrable functions from $\mR^n$ to $\mR^m$, ${\cal L}_\infty(\mR^n,\mR^m)$, bounded functions from $\mR^n$ to $\mR^m$, ${\cal C}^k(\mR^n,\mR^m)$ space of functions with $k$ continuous derivative, ${\cal C}^k_c(\mR^n,\mR^m)$, space of functions in ${\cal C}^k(\mR^n,\mR^m)$ with compact support, and ${\cal C}_0$ closure of ${\cal C}^k_c(\mR^n,\mR^m)$ in ${\cal L}_\infty$ norm. If the space $\mR^m$ is not specified then it is understood that the underlying space is $\mR$.
Let $B_\delta$ be the small neighborhood of the origin for some fixed $\delta>0$  and  $\bS:=\mR^n\setminus B_\delta$. We use ${\cal B}(\bS)$ to denote the Borel $\sigma$-algebra on $\bS$ and ${\cal M}(\bS)$ is the vector space of real-valued measure on ${\cal B}(\bS)$. ${\cal S}:= {\cal L}_1(\bS)\cap {\cal C}^2(\bS,\mR_{\geq 0})$.
\subsection{Perron-Frobenius and Koopman Operator for Stochastic System} 
The results and discussion in this section are taken from different references \cite{Lasota,shi2014fokker,manca2008kolmogorov,rudnicki2002markov} and is presented here for the sake of completeness. 
The theory of linear operators involving P-F and Koopman operators generalizes from deterministic dynamical systems to stochastic systems. 
Following additional assumptions are made on the vector field ${\bf F}$ and $\bn$ to ensure that these linear operators and their solutions are well defined.
\begin{assumption}\label{assume_diff} We assume that the vector field, ${\bf F}(\bx)$, and coefficients $a_{ij}(\bx):=\sum_{k=1}^q \bn_{ik}(\bx)\bn_{kj}(\bx)$ are ${\cal C}^4$ functions of $\bx$. Furthermore, following growth conditions are satisfied by the coefficients
{\small
\begin{eqnarray*}
|a_{ij}(\bx)|\leq M (1+|\bx|^2),|\tilde f_i(\bx)|\leq M(1+|\bx|),|\tilde c(\bx)|\leq M,
\end{eqnarray*}}
$\;\;i,j=1,\ldots, n$, where $M$ is some positive constant, ${\bf F}=(F_1,\ldots, F_n)^\top$ and 
\begin{eqnarray}\tilde F_i(\bx)=-F_i(\bx)+\sigma^2\sum_{j=1}^n\frac{\partial a_{ij}(\bx)}{\partial \bx_j},\\ \tilde c(\bx)=\frac{\sigma^2}{2}\sum_{i,j=1}^n\frac{\partial^2 a_{ij}(\bx)}{\partial x_i\partial x_j}-\sum_{i=1}^n\frac{\partial F_i(\bx)}{\partial x_i}.
\end{eqnarray}
Notice that $a_{ij}=a_{ji}$ by definition. For any given $\lambda=(\lambda_1,\ldots, \lambda_n)\in \mR^n$, we assume that uniform parabolicity assumptions holds if there exists a constant $\rho>0$ such that
\[\sum_{i,j=1}^n a_{ij}\lambda_i\lambda_j\geq\rho \sum_{i=1}^n \lambda_i^2\]
\end{assumption}

\begin{definition}[Koopman operator and  generator]\label{definition_koopmangenerator}  Let $\varphi(\bx)\in {\cal L}_\infty(\mathbb{R}^n)$, then the Koopman operator $\mU_t:{\cal L}_\infty(\mathbb{R}^n)\to {\cal L}_\infty(\mathbb{R}^n)$ for stochastic dynamical system (\ref{sys}) is defined as 
\begin{eqnarray}
[\mU_t \varphi](\bx)=\mE_\bx\left[\varphi(\bX_t^\bx)\right]\label{Koopman_operator}
\end{eqnarray}
$\mU:=\{\mU_t\}_{t\geq 0}$ restricts to strongly continuous semi-group on ${\cal C}_0$ and hence it has a infinitesimal generator ( \cite{kallenberg2006foundations}, Theorem 21.11). For any $\varphi\in {\cal C}_c^2(\mR^n)$, the infinitesimal generator is given by \footnote{The superscript notation of ${\cal A}^{{\bf f}_n}_K$ is used to signify the fact the the Koopman generator corresponds to the stochastic dynamical system $\dot \bx={\bf f}(\bx)+\sigma \bn(\bx)\xi$}
\begin{eqnarray}\label{eq:Koopmangenerator}
{\cal A}^{{\bf F}_n}_K\varphi:=\lim_{t\to 0}\frac{(\mU_t -I)\varphi}{t}\nonumber\\={\bf F}(\bx)\cdot \nabla  \varphi  +\frac{\sigma^2}{2}\sum_{i,j=1}^n [\bn \bn^\top]_{i j} \frac{\partial^2 \varphi}{\partial x_i\partial x_j}.
\end{eqnarray}

\end{definition}

We assume that the distribution of $\bX_0$ is absolutely continuous and has density $\rho_0({\bx})$. Then we know that ${\bX}_t$ has a density $\rho({\bx},t)$ i.e., 
\[{\rm Prob}\{\bX_t\in B\}=\int_B \rho(\bx,t)d\bx\]which satisfies the following Fokker-Planck (F-P) equation also known as Kolomogorov forward equation
\begin{eqnarray}
\frac{\partial \rho({\bx},t)}{\partial t}=-\nabla \cdot ({\bf F}(\bx) \rho)  +\frac{\sigma^2}{2}\sum_{i,j=1}^n \frac{\partial^2[(\bn \bn^\top)_{ij}\rho]}{\partial x_i\partial x_j}.
\end{eqnarray}
Following Assumption~\ref{assume_diff}, we know the solution $\rho({\bx},t)$ to F-P equation exists and is differentiable (Theorem 11.6.1  \cite{Lasota}). Under some regularity assumptions on the coefficients of the F-P equation (Definition 11.7.6  \cite{Lasota}) it can be shown that the F-P admits a unique classical solution (Theorem 11.7.1 \cite{Lasota}) given by 
\[\rho(\bx,t)=\int \Gamma(\bx,\by,t)f(\by)d\by,\] where $f\in {\cal C}^0(\mR^n)$ and satisfies $|f(\bx)|\leq ce^{\alpha \bx^2}$. The fundamental solution or kernel function, $\Gamma(\bx,\by,t)$ is defined for all $t>0$ and $\bx,\by\in \mR^n$, it is differetiable w.r.t. $t$ and twice differentiable w.r.t. $\bx,\by$.  The classical solution is used to define family of operators. Given any initial density function $\rho(\bx,0)=h(\bx)\in {\cal L}_1(\mR^n)$, we can define family of operators $\{\mP_t\}_{t\geq 0}$  by
\begin{eqnarray}
\rho(\bx,t)=[\mP_t h](\bx)=\int \Gamma(\bx,\by,t)h(\by)d\by.\label{pf_operator}
\end{eqnarray}

\begin{definition}[P-F operator and generator] The family of operators $\mP:=\{\mP_t\}_{t\geq 0}$ as defined in (\ref{pf_operator}) are called stochastic P-F semi-group. Following \cite{Lasota} (Theorems 11.6.1, 11.7.1, Corollary 11.8.1, and Remark 11.8.1) it also follows that the right hand side of the  F-P equation is also the infinitesimal generator for the stochastic P-F semi-group. In particular, we have
\begin{eqnarray}{\cal A}^{{\bf F}_{n}}_{PF}\psi :=\lim_{t\to 0}\frac{(\mathbb{P}_t-I)\psi}{t}\nonumber\\=-  \nabla \cdot ({\bf F}(\bx) \psi)  +\frac{\sigma^2}{2}\sum_{i,j=1}^n \frac{\partial^2[(\bn \bn^\top)_{ij}\psi]}{\partial x_i\partial x_j} .
\end{eqnarray}
\end{definition}
Since $\mP_t$ is a semi-group with generator ${\cal A}^{{\bf f}_n}_{PF}$, it satisfies
\begin{eqnarray}
\frac{d}{dt} \mP_t \psi={\cal A}^{{\bf F}_n}_{PF}\mP_t \psi\label{FP_diff}.
\end{eqnarray}
The duality between the P-F and Koopman generators can be expressed as follows (\cite{Lasota} Theorem 11.6.1):
\begin{eqnarray}
\int_{\mR^n} [{\cal A}^{{\bf F}_n}_K \varphi](\bx)\psi(\bx)d\bx=\int_{\mR^n}  \varphi(\bx)[{\cal A}^{{\bf F}_n}_{PF}\psi](\bx)d\bx.
\end{eqnarray}
Following \cite{manca2008kolmogorov}, the duality between the two semi-groups can be shown as follows. The adjoint to the Koopman semi-group, i.e., $\mU_t^\star$ can be defined using following relationship
\begin{eqnarray}
\left<\mU_t \varphi,\psi\right>&:=&\int_{\mR^n} [\mU_t \varphi](\bx)\psi(\bx)d\bx\nonumber\\=\int_{\mR^n} \varphi(\bx)[\mU_t^\star\psi](\bx)d\bx
&=&\left<\varphi,\mU_t^\star\psi\right> \label{duality_semigroup}
\end{eqnarray}
We claim that $\mU^\star_t=\mP_t$. We have
{\small
\[\frac{d}{dt}\left<\varphi,\mU_t^\star\psi\right>=\frac{d}{dt}\left<\mU_t\varphi,\psi\right>=\lim_{h\to 0}\frac{1}{h}\left<\frac{(\mU_{t+h}-\mU_t)}{h}\varphi,\psi \right>\]
\[=\lim_{h\to 0}\frac{1}{h}\left<\mU_t\frac{(\mU_h \varphi-\varphi)}{h},\psi\right>=\lim_{h\to 0}\frac{1}{h}\left<\frac{(\mU_h \varphi-\varphi)}{h},\mU_t^\star\psi\right>\]
\[=\left<{\cal A}^{{\bf F}_n}_K\varphi,\mU_t^\star\psi\right>=\left<\varphi,{\cal A}^{{\bf F}_n}_{PF}\mU_t^\star\psi\right>\]}
where we have used the fact that 
\[\lim_{h\to 0}\frac{(\mU_h -I)\varphi}{h}={\cal A}_K^{{\bf F}_n}\varphi,\;\;\left<{\cal A}_K^{{\bf F}_n}\varphi,\psi\right>=\left<\varphi,{\cal A}_{PF}^{{\bf F}_n}\psi\right>.\]
Hence, we obtain
\[\frac{d}{dt}\left<\varphi,\mU_t^\star\psi\right>=\left<\varphi,{\cal A}^{{\bf F}_n}_{PF}\mU^\star_t \psi\right>.\]
Since the above is true for all bounded function $\varphi\in {\cal C}_c^2(\mR^n)$ with compact support, we have 
\[\frac{d}{dt}\mU_t^\star\psi={\cal A}_{PF}^{{\bf F}_n}\mU_t^\star \psi\] and from the uniqueness of solution to the  F-P equation and from (\ref{FP_diff}) it follows that $\mU^\star_t=\mP_t$. Hence we have 
\begin{eqnarray}
\int_{\mR^n} [\mU_t \varphi](\bx)\psi(\bx)d\bx=\int_{\mR^n}  \varphi(\bx)[\mP_t\psi](\bx)d\bx.
\label{duality_operator}\end{eqnarray}








\begin{property}\label{property_positive}
Both the Koopman and P-F operators for the stochastic system are positive i.e., for any positive valued function $\psi(\bx)\geq 0$ and $\varphi(\bx)\geq 0$, we have 
\[[\mU_t\psi](\bx)\geq 0,\;\;\;[\mP_t\varphi](\bx)\geq 0.\]
\end{property}

\section{Stochastic Stability and Stabilization}\label{section_stochasticstability}
In this section, we provide the definition of almost everywhere almost sure stability and prove results providing necessary and sufficient condition for this notion of stability. 


\begin{definition}[Almost everywhere (a.e.) almost sure stability] \label{def_aestability} The equilibrium point at $\bx=0$ for system (\ref{sys}) is said to be  a.e. almost sure stable with respect to finite measure $\mu\in {\cal M}(\mR^n)$  if for almost all initial condition $\bx$ w.r.t. measure $\mu$, we have
\begin{eqnarray}
{\rm Prob}\{\lim_{t\to \infty}\bX_t^\bx=0\}=1.
\end{eqnarray}

\end{definition}

\begin{definition}[Local almost sure (a.s.) asymptotically stable]\label{localasasystable}The equilibrium point at $\bx=0$ for system (\ref{sys}) is said to be locally a.s. asymptotically stable if there exists a neighborhood $\cal N$ of the origin such that for all $\bx\in {\cal N}$, we have
\begin{eqnarray}
&{\rm Prob}\{\bX_t^\bx\in {\cal N}\}=1 \;\;\forall t\geq 0, \;\;\nonumber\\
&{\rm Prob}\{\lim_{t\to \infty}\bX_t^\bx=0\}=1.
\end{eqnarray}
\end{definition}


\begin{definition}[Equivalent measures]
Two measures $\mu_1\in {\cal M}(\mR^n)$ and $\mu_2\in{\cal M}(\mR^n)$ are said to be equivalent i.e., $\mu_1 \approx \mu_2$ provided $\mu_1(B)=0$ if and only if $\mu_2(B)=0$ for all set $B\in {\cal B}(\mR^n)$. 
\end{definition}
To prove the results providing necessary and sufficient condition for a.e. almost sure stability for system  we make following assumption on the system dynamics (\ref{sys}). 
\begin{assumption}\label{local_stability1}
 We assume that the equilibrium point $\bx=0$ for system (\ref{sys}) is locally a.s. asymptotically stable (Definition \ref{localasasystable}). 
\end{assumption}


\begin{theorem}\label{lemma1}
For system (\ref{sys}) satisfying Assumption \ref{local_stability1}, the $\bx=0$ is a.e. a.s. stable with respect to measure $\mu$ (Definition \ref{def_aestability}) if and only if
\begin{eqnarray}
\lim_{t\to \infty}[\mP_t h](\bx)=0
\end{eqnarray}
where, $h=\frac{d \mu}{d\bx}$ is assumed to be the Radon-Nikodym derivative of $\mu$ and $\mathbb{P}_t$ is stochastic P-F semi-group for system (\ref{sys}).
\end{theorem}
The proof of this theorem is provided in the Appendix.

\section{Main Results}\label{section_main}
In this section, we present the main results of this paper on the convex formulation to the optimal control problem. We consider optimal control problem for affine in control dynamical system of the form
\begin{eqnarray}
\dot \bx=\bar {\bf f}(\bx)+{\bf g}(\bx)\bar u+\sigma\bn(\bx)\xi\label{cont_syst1}
\end{eqnarray}
where, $\bx\in \mathbb{R}^n$ is the state, $\bar u\in \mathbb{R}$ is the control input. For the simplicity of presentation we present the results for the case of single input. The results generalize to multi-input case in a straight forward manner. The $\bar {\bf f}, \bg$, and $\bn$ vector fields are assumed to satisfy Assumption \ref{assume_diff}. Furthermore, we make following stabilizability assumption on the pair $(\bar{\bf f},\bg)$. 
\begin{assumption}\label{assumption_stabilizability}
We assume that the linearization of system dynamics at the origin $\bx=0$ i.e., $(\frac{\partial \bar {\bf f}}{\partial \bx}(0),\bg(0))$ is stabilizable. 
\end{assumption}
Using the above stabilizibility assumption, we can design a local stable controller using time series data. The detailed procedure for the design of such controller is given in Section \ref{local_control_data}. Let $u_\ell$ be the locally stabilizing  controller. Now defining ${\bf f}(\bx):=\bar {\bf f}(\bx)+\bg(\bx) u_\ell$ and $u=\bar u-u_\ell$, we can rewrite control system (\ref{cont_syst1}) as follows
\begin{equation}\label{cont_syst10}
    \dot \bx={\bf f}(\bx)+{\bf g}(\bx) u+\sigma\bn(\bx)\xi 
\end{equation}

\begin{remark}\label{remark_localstability}
Following is true for the  stochastic control  system (\ref{cont_syst10}).  With the control input $u=0$, the origin of system \eqref{cont_syst10} is locally a.s. asymptotically stable in small neighborhood of the origin, $\cal N$.  \cite[Proposition 4.1]{bardi2005almost}.
\end{remark}
 
Let $B_\delta$ be the small neighborhood of the origin for some fixed $\delta>0$ such that $B_\delta\subset {\cal N}$ and that $\bS:=\mR^n\setminus B_\delta$.


\subsection{Convex Formulation to the Stochastic OCP Using P-F operator}
Consider the stochastic OCP, where the objective is to minimize the cost function of the following form
\begin{eqnarray}
&J^\star(\mu_0)=\inf_u\int_{\bS}\mathbb{E}_{\bx}\left[\int_0^\infty q(\bX_t^\bx)+ r u_t^2 \;dt\right] d\mu_0(\bx)\label{costfunction}\\
&{\rm s.\;t.\; (\ref{cont_syst10})}\nonumber
\end{eqnarray}
for some $r>0$. The $\mathbb{E}_\bx$ stands for the expectation taken w.r.t. different realization of stochastic trajectories starting from initial condition $\bx$. We now make 
following assumption on the state cost function $q(\bx)$ and the measure $\mu_0$.
\begin{assumption}\label{assume_costfunction}
We assume that the state cost function $q: \mR^n\to \mR^{+}$ is zero at the origin and uniformly bounded away from zero outside the neighborhood $\cal N$. Furthermore, the measure $\mu_0$ is assumed to be equivalent to Lebesgue with Radon-Nikodym derivative $h_0$ i.e., $\frac{d\mu_0}{d\bx}=h_0(\bx)\in {\cal L}_1(\mR^n)\cap {\cal C}^2(\mR^n,\mR_{\geq 0})$. 
\end{assumption}
Some comments are necessary for the nature of the cost function chosen for the stochastic optimal control. Define
\[V(\bx)=\mE_\bx\left[\int_0^\infty q(\bX_t^\bx)+ ru_t^2\;dt\right]\]
where the expectation are taken along different realization of stochastic trajectory starting from initial condition $\bx$. The  $V(\bx)$ is the cost function considered in traditional stochastic optimal control problem \cite{stengel1986stochastic}. Our proposed cost function in (\ref{costfunction}) can be expressed in terms of  $V(\bx)$ as
\begin{align}J(\mu_0)=\int_{\bS}V(\bx)d\mu_0(\bx)\label{cost1}
\end{align}
So the proposed cost function is a weighted form of the traditional cost function, $V(\bx)$, where the weights are determined by the measure $\mu_0$. Another important distinction is that the cost function in (\ref{costfunction}) is minimized over the set $\bS$ excluding the small neighborhood of the origin. We clarify the reason for this exclusion in Section \ref{section_localoptimal}. 
To ensure that the cost function in (\ref{cost1}) is finite  we introduce following definition of admissible control and initial measure $\mu_0$. 
\begin{definition} For system (\ref{cont_syst10}), a feedback controller $u=k(\bx)$ is said to be admissible if $k(0)=0$, $k\in {\cal C}^4(\mR^n)$ , $k(\bx)$ a.e. almost sure stabilize the origin, and 
\begin{align}
    V(\bx)=\mE_\bx\left[\int_0^\infty q(\bX_t^\bx)+rk(\bX^\bx_t)dt\right]\leq M \parallel \bx\parallel^{2\bar \gamma}
\end{align}
for some integer $\bar\gamma\geq 1$ and a.e. w.r.t. Lebesgue measure $\bx$. For a given admissible control, the initial measure $\mu_0$ or the associated density function, $h_0(\bx)$, is said to be admissible if 
\begin{align}
    \int_\bS V(\bx)d\mu_0(\bx)=\int_\bS V(\bx)h_0(\bx)d\bx<\infty\label{integrable}
\end{align}
\end{definition}
\begin{remark}
For a given admissible control, it is always possible to choose admissible density function, $h_0(\bx)$. For example, we can choose $h_0(\bx)=\frac{1}{\parallel\bx \parallel^{2(\bar\gamma+1)}}$ to ensure condition (\ref{integrable}) is satisfied. 
\end{remark}

The following assumption is made on the solution of the optimal control problem. 

\begin{assumption}\label{assume_OCP}
We assume that there exists a admissible feedback controller and that the optimal input is feedback in nature i.e., $u^\star=k^\star(\bx)$, with the function $k$ assumed to be ${\cal C}^4$ function of $\bx$.   
\end{assumption}

With the assumed feedback form of the control input, the SOCP can be written as 
\begin{eqnarray}
 \inf\limits_{k\in {\cal C}^4} &\int_{\bS}\mathbb{E}_\bx\left[\int_0^\infty q(\bX_t^\bx)+ r k(\bX_t^\bx)^2 \;dt\right] d\mu_0(\bx)\nonumber\\
 {\rm s.t.}&\dot \bx={\bf f}(\bx)+\bg(\bx)k(\bx)+\sigma\bn(\bx)\xi\label{ocp_main}
\end{eqnarray}

Following is one of the main results of this paper. 
\begin{theorem}\label{Theorem_mainPF}
Consider the optimal control problem (\ref{ocp_main}), with the system dynamics, cost function, and optimal control satisfying Assumptions \ref{assumption_stabilizability}, \ref{assume_costfunction}, and \ref{assume_OCP} respectively. The OCP (\ref{ocp_main}) can be written as following infinite dimensional convex optimization problem 
\begin{eqnarray}
J^\star(\mu)=\inf_{\rho\in {\cal S},\bar \brho\in {\cal C}^4} \;\;\; \int_{\bS} q(\bx)\rho(\bx)+r \frac{\bar \rho(\bx)^2}{\rho(\bx)} d\bx\nonumber\\
{\rm s.t}.\;\;\;\nabla\cdot ({\bf f}\rho +\bg\bar \rho)-\frac{\sigma^2}{2}\sum_{i,j=1}^n \frac{\partial^2[(\bn \bn^\top)_{ij}\rho]}{\partial \bx_i\partial \bx_j}=h_0.
\label{eqn_ocp1}
\end{eqnarray}
where $ {\cal S}:={\cal L}_1(\bS)\cap {\cal C}^2(\bS,\mR_{\geq 0})$.
The optimal feedback control input is recovered from the solution of the above optimization problem as  
\begin{eqnarray}
k(\bx)=\frac{\bar \rho(\bx)}{\rho(\bx)}.
\end{eqnarray}
Furthermore, the optimal control $k(\bx)$ is  a.e. a.s.  stabilizing the origin. 
\end{theorem}
\begin{proof}
\begin{eqnarray}
J(\mu)=\int_{\bS}\mathbb{E}_{\bx}\left[\int_0^\infty q(\bX_t^\bx)+ u_t^2 \;dt\right] d\mu_0(\bx)
\end{eqnarray}
The expectation can be moved inside  the integral by using Fubini's theorem \cite{kallenberg2006foundations} as the cost is positive and finite following Assumption \ref{assume_OCP}. Additionally, we consider the fact that $d\mu_0=h_0 d\bx$ to obtain
\[\int_{\bS}\int_0^\infty \mE_\bx \left[q(\bX_t^\bx)+u_t^2\right] dt h_0(\bx) d\bx.\]
Using Assumption (\ref{assume_OCP}) on the existence of feedback input for which the cost function is finite, we obtain after substituting $u=k(\bx)$
\begin{eqnarray}J=\int_{\bS}\int_0^\infty \mE_\bx[\varphi(\bX_t^\bx)]dt h_0(\bx)d\bx\nonumber\\=\int_{\bS}\int_0^\infty [\mU_t^c\varphi](\bx)dt h_0(\bx)d\bx,\end{eqnarray}
where we defined $\varphi(\bx):=q(\bx)+r k(\bx)^2$ and $\mU_t^c$ and $\mP_t^c$ are notation for the Koopman and P-F semi-groups for the closed loop system ${\bf f}_c:={\bf f}+\bg k+\sigma\bn \xi$.
Using the linear property of the Koopman operator and the duality of Koopman and P-F semi-groups we obtain
\begin{eqnarray}
J=\int_{\bS}\int_0^\infty \varphi(x)[\mP_t^c h_0](\bx)dt d\bx=\int_{\bS}\varphi(\bx)\int_0^\infty[\mP_t^c h_0](\bx) d\bx.\nonumber\\\label{ss2}
\end{eqnarray}
Using the Assumption \ref{assume_costfunction} on the uniform lower bound, say $\kappa_0$, for state cost function outside $\cal N$, and Assumption \ref{assume_OCP} on the finite value of cost for the feedback control input , we obtain
\begin{eqnarray}\kappa_0 \int_{\bS}\int_0^\infty[\mP_t^ch_0](\bx)dt d\bx\leq \int_{\bS}\varphi(\bx)\int_0^\infty[\mP_t^ch_0](\bx)dtd\bx\leq M\nonumber\label{bound}\\
\end{eqnarray}
for some constant $M<\infty$. Define 
\begin{align}
\rho(\bx):=\int_0^\infty [\mP_t h_0](\bx) dt \label{rhointegrable}
\end{align}
From (\ref{bound}) it follows that $\rho(\bx)$ is well defined for almost all $\bx\in \bS$ and that $\rho(\bx)$ is integrable function on $\bS$. 
Following the definition of P-F operator and the differentiability assumption made on $h_0$ (Assumption \ref{assume_costfunction}) it follows that the $[\mP_t^c h_0](\bx)$ is absolutely continuous w.r.t. time and hence using Barbalat Lemma it follows that 
\begin{eqnarray}
\lim_{t\to \infty}[\mP_t^c h_0](\bx)=0
\end{eqnarray}
for almost all $\bx \in \bS$. 
We next show that $\rho(\bx)$ satisfy  the following equation
\begin{equation}
\nabla\cdot (({\bf f}_c(\bx) \rho(\bx))-\frac{\sigma^2}{2}\sum_{i,j=1}^n \frac{\partial^2[(\bn \bn^\top)_{ij}\rho]}{\partial \bx_i\partial \bx_j}=h_0(\bx),\label{steady_pde}
\end{equation}
for a.e. $\bx\in \bS$.
Substituting the integral formula for $\rho(\bx)$ in (\ref{steady_pde}), we obtain
\begin{eqnarray}
&&\nabla\cdot ({\bf f}_c(\bx) \rho(\bx))-\frac{\sigma^2}{2}\sum_{i,j=1}^n \frac{\partial^2[(\bn \bn^\top)_{ij}\rho]}{\partial \bx_i\partial \bx_j}\nonumber\\&=&\int_0^\infty \nabla\cdot ({\bf f}_c(\bx)[\mathbb{P}_t^c h_0](\bx)) dt\nonumber\\
&-&\int_0^\infty\frac{\sigma^2}{2}\sum_{i,j=1}^n \frac{\partial^2[(\bn \bn^\top)_{ij}[\mathbb{P}_t^c h_0](\bx)]}{\partial \bx_i\partial \bx_j}dt\nonumber\\&=&\int_0^\infty -\frac{d}{dt}[\mathbb{P}^c_t h_0](\bx)dt=-[\mathbb{P}^c_th_0](\bx)\Big|^{\infty}_{t=0}=h_0(\bx)\label{eq11}
\end{eqnarray}
where we have used the infinitesimal generator property of P-F operator(i.e.,  Eq. (\ref{FP_diff})) and the fact that $\lim_{t\to \infty} [\mathbb{P}_t^c h_0](\bx)=0$.
Furthermore, since $h_0> 0$, it follows that $\rho> 0$ from the positivity property of P-F semi-group $\mathbb{P}_t^c$. 
Combining (\ref{ss2}) and (\ref{eq11}) and by defining $\bar \rho(\bx):=\rho(\bx) k(\bx)$, it follows that the SOCP problem can be written as  convex optimization problem (\ref{eqn_ocp1}). The optimal control $k^\star(\bx)$ obtained as the solution of optimization problem (\ref{eqn_ocp1}) is
a.e. uniform stochastic stabilizing follows from the fact the optimal solution $\rho^\star(\bx)$ is integrable and argument following Eq. (\ref{rhointegrable}). The optimal solution  $\rho^\star(\bx)\in {\cal S}$ follows from the fact that $h_0\in {\cal L}_1(\mR^n,\mR_{> 0})\cap {\cal C}^2(\mR^n)$ from Assumption \ref{assume_costfunction}, the definition of $\rho$ Eq. (\ref{rhointegrable}) and the P-F operator (\ref{pf_operator}).

\end{proof}

\subsection{Convex Formulation Using P-F operator with input and state constraints}
It is possible to incorporate the hard constraints on the control input and the state in the convex formulation of SOCP as discussed in the previous section.  These hard constraints can be written convexily in the SOCP formulation. Assume that the control input need to satisfy $|u|\leq M_1$ for some constant $M_1$. Using the fact that the control input is feedback and of the form $u=k(\bx)=\frac{\bar \rho^2(\bx)}{\rho(\bx)}$, the hard constraints on the control input can be written convexily in terms of optimization variable $\rho$ and $\bar \rho$ as 
\[\bar \rho^2(\bx)\leq M_1 \rho(\bx).\]
For state constraints, we are interested in restricting the states to remain in a certain region of the state space, say ${\cal R}$. Since the system is stochastic, the constraints can be imposed in the expectation and  can be written as 
\[\int_\bS \mE_\bx\left[\int_0^\infty  1_{{\cal R}c}(\bX_t^\bx) dt\right]d\mu_0(\bx)=0\]
where $1_{{\cal R}_c}(\bx)$ is a indicator function of set ${\cal R}_c$, and ${\cal R}_c$ is complement of set ${\cal R}$. The above condition captures that the expected value of trajectory entering in the set $\cal R$ is zero starting from initial condition $\bx$. The initial conditions are distributed according to the initial measure $\mu_0$ as defined in the SOCP formulation. The state constraints can be written convexly using the following results. 
\begin{lemma}
For the optimal control problem (\ref{ocp_main}) under Assumptions \ref{assumption_stabilizability}, \ref{assume_costfunction}, and \ref{assume_OCP}, we have
\begin{align}
 \int_\bS \mE_\bx\left[\int_0^\infty  1_{{\cal R}_c}(\bX_t^\bx) dt\right]d\mu_0(\bx)=\int_{\bS} 1_{{\cal R}_c}(\bx) \rho(\bx)d\bx   
\end{align}

\end{lemma}
\begin{proof}
The proof follows by using the duality between the Koopman and the P-F operator and from the Definition of density function $\rho$ (Eq. (\ref{rhointegrable})).  
\end{proof}
Using the results of the above theorem the state constraint can be written convexily in terms of optimization variable $\rho$ as
\begin{align}\int_{\bS} 1_{{\cal R}_c}(\bx) \rho(\bx)d\bx=0 
\end{align}
In \cite{safetyPF}, we have used these state constraints in the convex formulation of optimal navigation problems with obstacle avoidance. 

\subsection{Convex Approach to Stochastic Stabilization Using P-F operator}
The convex formulation to the stochastic stabilization will arise as a particular case of the proposed convex formulation to the stochastic optimal control problem. In particular, the constraints in the optimization problem (\ref{eqn_ocp1}) will lead to the stability of the origin for the stochastic system. The convex formulation to the stochastic stabilization also appears in the work of \cite{van2006almost}. However, the proposed data-driven approach for stochastic stabilization is new to this paper. The stochastic a.e. stabilization problem w.r.t. measure $\mu_0$ can be posed as following feasibility problem

\begin{eqnarray}
\nabla\cdot({\bf f}\rho+\bg \bar\brho)-\frac{\sigma^2}{2}\sum_{i,j=1}^n \frac{\partial^2[(\bn \bn^\top)_{ij}\rho]}{\partial \bx_i\partial \bx_j}=h_0.\label{stabilization_constraints}
\end{eqnarray}
where $h_0$ is the density corresponding to the measure $\mu_0$. 
The stabilizing feedback controller can be recovered as $\bk(\bx)=\frac{\bar \brho(\bx)}{\rho(\bx)}$. In our computational section, we outline a procedure for the constructing the finite dimensional approximation of the stability constraints (\ref{stabilization_constraints}) from the time-series data.


\subsection{Local Optimal Controller}\label{section_localoptimal}
The density function $\rho$ for the solution of optimization problem satisfy following integral formula
\[\rho(\bx)=\int_0^\infty [\mathbb{P}^c_t h_0](\bx) dt\]
where $\mathbb{P}_t^c$ is the P-F operator for the closed-loop system $\dot \bx={\bf f}(\bx)+{\bf g}(\bx){k}(\bx)+\bn(\bx)\xi$ and hence $\rho$ serves as an occupancy measure i.e., $\int_A \rho(\bx) d\bx=\left<\int_0^\infty[\mathbb{U}_t \chi_A] dt,h\right>$ signifies the amount of time closed-loop system trajectories spend in the set $A$ with initial condition supported w.r.t. measure $\mu_0$. Because of this, $\rho(\bx)$ has singularity at the equilibrium point  stabilized by the closed-loop system as all the trajectories are funnel through the neighborhood of the origin to the origin. Due to  this singularity at the origin, we need to exclude the small neighborhood around the origin for the proper parameterization of the density function $\rho$ in the computation of optimal control. In particular, the optimization problem (\ref{eqn_ocp1})  is solved excluding the small neighborhood around the origin. The local optimal control is obtained based on the linearization of nonlinear dynamics at the origin.  From this local control we determine a local Lyapunov function $\bP$ and define a local density function as
\begin{eqnarray}
\rho_L(\bx)= \text{max}\{(\bx^T \bP \bx)^{-3}-\gamma, 0\}, \label{localdensity}
\end{eqnarray}
which is used for applying general control strategy. We combine the local and global controllers through the blending procedure as follows:

\begin{equation}
   \bar u(\bx)=\frac{\rho_L(\bx)}{\rho_L(\bx) +\rho(\bx)} u_\ell(\bx) + \frac{\rho(\bx)}{\rho_L(\bx) +\rho(\bx)} u(\bx).  \label{Blendingctrl}
\end{equation}
Notice that with blending control, we can get a smooth control on the whole workspace.

\subsection{Stochastic OCP Using Koopman Operator}
In this section, we show how the Koopman operator can be used to develop optimal control results. The results involving the Koopman operator can be viewed as a dual to P-F-based approach for SOCP discussed in the previous section. We show that the HJB equation can be viewed from the perspective of Koopman operator theory. The results from this section will be used to develop data-driven optimal control based on the approximation of the Koopman operator. Consider the following optimal control problem
\begin{eqnarray}
V^\star(\bx)=\inf_u \;\mathbb{E}_{\bx}\left[\int_0^\infty q(\bX_t^\bx)+ r u_t^2 \;dt\right]\nonumber\\
{\rm s.t.}\;\dot\bx={\bf f}(\bx)+{\bf g}(\bx)u+\sigma \bn(\bx)\xi\label{ocp_koopman}
\end{eqnarray}
With some slight abuse of notation, we are assuming that the stochastic control dynamical system in (\ref{ocp_koopman}) satisfies the same assumption as stated for Eq. (\ref{cont_syst1}).
We make following assumption.
\begin{assumption}\label{assume_ocpkoopman}
The state cost function is assumed to be radially unbounded  and positive i.e., $q(\bx)\to \infty$ as $\bx \to \infty$, $q(\bx)\geq 0$, $q(0)=0$, and $q\in {\cal C}^2(\mR^n)$. We assume that there exists a feedback control input $u=k(\bx)\in{\cal C}^4(\mR^n)$ for which the cost function is finite for any finite $\bx$ and the optimal control input is feedback in nature $u^\star =k^\star(\bx)\in {\cal C}^4(\mR^n)$.
\end{assumption}
The above assumption could be restrictive as it rules out systems that does not admit continuous control and hence optimal solution that are not continuous. It is known that the SOCP admits viscosity-based solution which is weaker than classical solution \cite{crandall1992user}. While the classical solution are differentiable in the domain of interest, the viscosity-based solution allows for a continuous function to be defined as a unique solution of HJB equation that do not admit continuous solution in a classical sense.
However, given the focus of this paper is on the use of operator theoretic framework for data-driven optimal control we proceed with this stronger assumption.
Following theorem is the main result expressing the solution to SOCP in terms of the Koopman generator.
\begin{theorem}\label{theorem_mainKoopman}
The solution to the optimal control problem (\ref{ocp_koopman}) satisfying Assumption \ref{assume_ocpkoopman} can be obtained by solving following HJB equation for $V^\star(\bx)\in {\cal C}^2(\mR^n,\mR_{\geq 0})$ expressed in terms of Koopman generator.

\begin{align}
{\cal A}_K^{{\bf f}_c} V^\star(\bx)=-q(\bx)-rk^\star(\bx)^2\label{hjb1}
\end{align}
with the optimal control $k^\star(\bx)$ given by 
\begin{align}
    k^\star(\bx)=-\frac{1}{2}r^{-1}\bg(\bx)\cdot \nabla V^\star(\bx)\label{hjb2}
\end{align}
with $V^\star(0)=0$. ${\cal A}_K^{{\bf f}_c}$ in (\ref{hjb1}) is the infinitesimal generator of the Koopman operator corresponding to the closed loop stochastic dynamics, $\dot \bx={\bf f}(\bx)+\bg(\bx)k(\bx)+\bn(\bx)\xi=: {\bf f}_c(\bx,\xi)$ (Definition \ref{definition_koopmangenerator}) and given by
\[{\cal A}_K^{{\bf f}_c}:=\left({\bf f}+\bg k\right)\cdot \nabla V^\star+\frac{\sigma^2}{2}\sum_{i,j=1}^n [\bn \bn^\top]_{ij} \frac{\partial^2 V^\star}{\partial x_i\partial x_j}\]
\end{theorem}
\begin{proof}
Moving the expectation inside the time integral we can write
\begin{eqnarray}
\mathbb{E}_{\bx}\left[\int_0^\infty q(\bX_t^\bx)+ r u_t^2 \;dt\right]=\int_0^\infty \mathbb{E}_{\bx}\left[ q(\bX_t^\bx)+ r u_t^2 \;dt\right]
\end{eqnarray}
Using the assumption that optimal control is feedback in nature and the definition of Koopman semi-group, we write
\begin{eqnarray}
V(\bx)=\int_0^\infty \mathbb{E}_{\bx}\left[ q(\bX_t^\bx)+ r u_t^2 \;dt\right]=\int_0^\infty [\mU_t^c \varphi](\bx)dt\label{integral_koopman}
\end{eqnarray}
where $\varphi:=q+r k^2$. Following Assumption \ref{assume_ocpkoopman}, we know that the cost function, $V(\bx)$ is finite for any finite $\bx$. Furthermore, $[\mU_t^c \varphi](\bx)$ is uniformly continuous w.r.t. time. The uniform continuity follows from the definition of the Koopman operator and the solution of the feedback system satisfy uniform continuity w.r.t. time property. Hence, we can apply Barbalat Lemma \cite{barbalat1959systemes} to conclude that $\lim_{t\to \infty}[\mU_t^c \varphi](\bx)=0$. 
 It can be shown that $V(\bx)$ satisfy following equation
\begin{eqnarray}
{\cal A}_K^{{\bf f}_c} V(\bx)=-\varphi(\bx)\label{koopman_pde}
\end{eqnarray}
To prove this we substitute (\ref{integral_koopman}) in (\ref{koopman_pde}) to obtain
\begin{eqnarray}
{\cal A}_K^{{\bf f}_c} V(\bx)=\int_0^\infty({\bf f}+\bg k)\cdot  \nabla[\mU^c_t \varphi](\bx)dt\nonumber\\
+\int_0^\infty \frac{\sigma^2}{2}\sum_{i,j=1}^n [\bn \bn^\top]_{ij} \frac{\partial^2 }{\partial x_i\partial x_j}[\mU_t^c\varphi](\bx)\nonumber \\
=\int_0^\infty \frac{d}{dt}[\mU^c_t \varphi](\bx)dt\nonumber\\
=[\mU_t^c \varphi](\bx)|_{t=0}^\infty
=-\varphi(\bx)
\end{eqnarray}
where we have used the fact that $\lim_{t\to \infty}[\mU_t^c \varphi](\bx)=0$. Since the objective is to find a feedback control input that minimizes the cost function $V(\bx)$, differentiating (\ref{koopman_pde}) w.r.t. $k$ and equating it to zero gives us the optimal feedback input 
\begin{eqnarray}k(\bx)=-\frac{1}{2}r^{-1}\bg\cdot \nabla V(\bx)\label{optimal_inputKoopman}
\end{eqnarray}
Substituting (\ref{optimal_inputKoopman}) in (\ref{koopman_pde}) we obtain the desired equation for the optimal cost function $V^\star(\bx)$ satisfying the HJB equation (\ref{hjb1}). The optimal $k^\star$ has the form (\ref{optimal_inputKoopman}) with $V(\bx)$ replaced with optimal $V^\star(\bx)$. Furthermore,  $V^\star(\bx)$ and $V^\star(0)=0$ follows from the fact that $\varphi(\bx)\in {\cal C}^2(\bX)$ with $\varphi(0)=0$ and the integral formula for $V(\bx)$ in (\ref{integral_koopman}).

\end{proof}

\begin{remark}
The optimal cost function $V^\star(\bx)$ from Theorem \ref{theorem_mainKoopman} can be viewed as dual to the optimal density function $\rho^\star(\bx)$ from Theorem (\ref{Theorem_mainPF}). While $V^\star(\bx)$ is obtained by the time integral of the cost function involving Koopman operator $\rho^\star(\bx)$ is obtained as time integral of density function $h_0(\bx)$ using the P-F operator. The duality between $V^\star(\bx)$ and $\rho^\star(\bx)$ then follows using the duality between the Koopman and P-F operators. 
Furthermore, just like $\rho^\star(\bx)$ ensures a.e. a.s. stabilization of the equilibrium point, the optimal cost function $V^\star(\bx)$ also serves as stability certificate for the closed loop stochastic system. In particular, $V^\star(\bx)$ guarantee a.s. global asymptotic stability of the origin  as it satisfies following inequality
\[{\cal A}_K^{{\bf f}_c} V(\bx)<0\]
with $V^\star(\bx)>0$ for $\bx \neq 0$ and $V^\star(\bx)$ is radially unbounded. The radially unbounded property follows from the assumption on $q(\bx)$ (Assumption \ref{assume_costfunction}), Eq. (\ref{integral_koopman}), and positivity property of the Koopman operator \cite{Hasminskii_book,Kushner_book}. This duality between density function and the Lyapunov function as stability certificates for deterministic and stochastic systems is studied in \cite{Rantzer01,Vaidya_TAC,van2006almost}. 
\end{remark}

The final result of Theorem \ref{theorem_mainKoopman} is not new; however, the interpretation of the HJB equation in terms of the generator of the Koopman operator is novel.
This connection between the Koopman generator for the stochastic system and the HJB equation allows us to provide a novel approach for the data-driven numerical solution of HJB equation based on the  finite-dimensional approximation of the Koopman operator. 
Following Assumtion \ref{assume_ocpkoopman}, we rule out the possibility of approximating the viscosity-based solution of the HJB equation. 

For the purpose of computation, the complexity associated with the nonlinear nature of the HJB equation can be broken down using the popular approach employed in Reinforcement Learning theory. In particular, using ideas from RL's generalized policy iteration (GPI) algorithm, an iterative process can be provided for solving the HJB equation. The working of GPI as it applies to solving the HJB equation is evident if we split the HJB equation as two equations in (\ref{hjb1})-(\ref{hjb2}).


Consider the $k^{th}$ step of the iteration and let  $k_k(\bx)$ be the a.s. globally stabilizing feedback controller. Then the value function $V_{k}$ can be obtained as the solution of the following linear equation expressed in terms of the Koopman generator of the stochastic feedback system. 
\begin{align}
    {\cal A}_K^{{\bf f}+\bg k_k}V_k(\bx)=-q(\bx)-rk_k(\bx)^2\label{step1}
\end{align}
For a given fixed feedback controller $k_k$, the above equation is a linear equation to be solve for unknown $V_k$. Once we solve for $V_k$ the feedback controller can be updated to its new value as 
\begin{align}
k_{k+1}(\bx)=-\frac{1}{2}r^{-1}\bg(\bx)\cdot \nabla V_k=-\frac{1}{2}r^{-1}{\cal A}_K^{\bg} V_k\label{step2}
\end{align}

The update step in Eq. (\ref{step1}) going from $k_k(\bx)\to V_k(\bx)$ corresponds to the policy evaluation step  whereas the step of going from $V_k(\bx)\to k_{k+1}(\bx)$ in (\ref{step2}) correspond to policy improvement step of GPI \cite{sutton2018reinforcement}. However, unlike GPI algorithm used in the execution these steps where the value function or cost function, $V_\ell$, is paramaterized and approximated from time-series data \cite{bertsekas2011approximate,sutton2018reinforcement}, our proposed computational scheme will rely on using time- series data for the approximation of the Koopman generator, ${\cal A}_K^{{\bf f}+\bg k_k}$ and ${\cal A}_K^{\bg}$ for the approximation of $V_k$ and $k_k$.  Note that the convergence analysis of the policy iteration is addressed for deterministic systems in  \cite{beard1997galerkin} and stochastic setting  \cite{santos2004convergence,bertsekas2015value}. We refer the readers to survey article on approximate policy iteration in \cite{bertsekas2011approximate} involving approximation of value function.

The spectrum of the Koopman operator carries essential information about the system dynamics. Our proposed Koopman-based perspective to the HJB equation will enable the discovery of a spectral-based method for optimal control. 


\section{Data-driven Approximation of Optimal Control}\label{section_data-driven}
This section will discuss a linear operator-based numerical scheme for the data-driven approximation of optimal control. The numerical scheme for data-driven optimal control will rely on the P-F-based convex formulation to the OCP as discussed in (\ref{eqn_ocp1}) and Koopman-based value iteration as given in Eqs. (\ref{step1})-(\ref{step2}). We start with the finite-dimensional approximation of the P-F and Koopman operators for stochastic control systems and then utilize these approximations to design optimal control and stabilizing feedback control. 

\subsection{Approximation of P-F Operator}
 For the finite dimensional approximation of the P-F based convex formulation of the OCP problem in Eq. (\ref{eqn_ocp1}), we need to approximate the P-F generator corresponding to dynamical systems 
$\dot \bx={\bf f}(\bx)+\sigma \bn(\bx)\xi$ and $\dot \bx=\bg(\bx)$ using the time series data generated from controlled stochastic system $\dot \bx={\bf f}(\bx)+\bg(\bx)u+\sigma \bn(\bx)\xi$.



For the approximation of P-F generator for $\dot \bx={\bf f}(\bx)+\sigma\bn(\bx)\xi$ we use control input $u=0$. 
Let $\{\bx_i,\by_i^\ell,u_i=0\}$ be the two consecutive snapshots obtained by simulating the system (\ref{sys}) using Euler-Maruyama discretization \cite{kloeden2013numerical}. In particular, we have 
\begin{equation}
\mathbf{y}_i^\ell=\bx_i + \Delta t{\bf f}(\bx_i) + \sigma\sqrt{\Delta t}\bn(\bx_i)\xi^\ell
\end{equation}
where $i=1,\ldots, N$ are the number of initial conditions assume to be uniformly distributed in the state space and $\ell=1,\ldots, R$ are the number of realization. For the finite dimensional approximation, let $\bPsi: \bX\to \mR^K$ be the vector valued  basis functions. We make following assumption on the choice of basis functions. 
\begin{assumption}\label{assume_basis}
We assume that the basis functions, $\psi_k(\bx)\in {\cal L}_2(\mR^n)$ for $k=1,\ldots, K$ are non-negative and linearly independent. We denote the set of basis functions as 
\[\bPsi(\bx)=[\psi_1(\bx),\ldots,\psi_K(\bx)]^\top.\]
\end{assumption}
Our objective is to construct the projection of the infinite-dimensional P-F operator on the finite-dimensional subspace spanned by $\{\psi_k(\bx)\}_{k=1}^K$. The projection should preserve the positivity and Markov properties of the P-F operator. 
Writing 
\[\phi(\bx)=\bPsi(\bx)^\top {\bf a},\;\;\;\hat \phi(\bx)=\bPsi(\bx)^\top \hat{\bf a}\]
The action of the Koopman operator (\ref{Koopman_operator}) on the finite dimensional basis function can be written as  
\begin{eqnarray}
\hat \phi(\bx_i)=\bPsi^\top (\bx_i)\hat {\bf a}=[\mU_t \phi](\bx_i)+r\nonumber\\=\mE_\bx[\phi(\bX_{\Delta t}^{\bx_i})]+r
\approx \frac{1}{R}\sum_{\ell=1}^R \bPsi^\top(\by_i^\ell)^\top {\bf a}+r\label{ss}
\end{eqnarray}
where $r$ is the residual term and the objective is to minimize the residual term. We have approximated the expectation above using multiple, $R$, realization of the system trajectories. 
For the approximation of the P-F operator preserving the positivity and Markov property, Naturally Structured Dynamic Mode Decomposition (NSDMD) algorithm is proposed in \cite{huang2018data}. The NSDMD algorithm can be viewed as the generaization of the popular Extended Dynamic Mode Decomposition (EDMD) algorithm to incorporate the positivity and Markov constraints.  Pre-multiplying (\ref{ss}) with $\bPsi(\bx_i)$ and summing over the initial conditions and the different realization to construct following matrices
\begin{eqnarray}
\bG:=\frac{1}{M}\sum_{i=1}^N \bPsi(\bx_i)\bPsi^\top (\bx_i)\nonumber\\\bA:=\frac{1}{MR}\sum_{i,\ell}^{N,R}\bPsi(\bx_i)\bPsi^\top (\by_i^\ell).\label{GA_compute}
\end{eqnarray}
The $\bG$ and $\bA$ matrices are used to formulate a least square problem for the minimization of residual term $r$ as follows: 
\begin{eqnarray}\label{nsdmd_algo}&\min_{\bK_0\in \mathbb{R}^{K\times K}}\parallel \bG \bK_0-\bA\parallel_F\\
\text{s.t.} \;\;
& [{\bf{\Lambda} {\bf K}_0\bf{\Lambda}^{-1}}]_{ij}\geq 0,\;\;\;\bf{\Lambda}{\bf K}_0\bf{\Lambda}^{-1}\mathds{1} = \mathds{1}\label{nsdmd2}
\end{eqnarray}
 where $\mathds{1}$ is a vector of all ones and ${\bf \Lambda}=\int \bPsi(\bx)\bPsi^\top(\bx)d\bx$ Note that the $\bf \Lambda$ matrix can be computed analytically as integral can be computed explicitly for Gaussian basis functions.  The constraints are imposed to ensure  the positivity and Markov property of the linear operator. Using the duality relation between the P-F and Koopman operators (\ref{duality_operator}), the P-F operator can be obtained as 
 \begin{eqnarray}\mP_{\Delta t}\approx\bP_0=\Lambda^{-1}\bK_0^\top \Lambda \label{P_0approx}
\end{eqnarray}
  In practice, we implement  following numerical efficient optimization problem which is equivalent to (\ref{nsdmd_algo})-(\ref{nsdmd2}) to compute the P-F operator. \begin{eqnarray}\label{nsdmd}
&\min\limits_{\hat{\bf P}_0}\parallel \hat{\bf G}{\hat{\bf P}_0}-\hat{\bf A}\parallel_F\\\nonumber
\text{s.t.} \;\;
& [\hat{\bf P}_0]_{ij}\geq 0,\;\;\;\hat {\bf P}_0\mathds{1} = \mathds{1},
\end{eqnarray}
where,
\begin{eqnarray}\hat{\bf G}={\bf G}{\bf \Lambda}^{-1},\;\;\hat{\bf A}={\bf A}{\bf \Lambda}^{-1}\label{hatGA},
\end{eqnarray}
where $\bf G$ and $\bf A$ are as defined in (\ref{GA_compute}). The P-F operator is obtained from $\hat \bP$ as 
\begin{eqnarray}
\bP_0=\hat \bP_0^\top.
\end{eqnarray}
Similarly, the infinitesimal generator for the P-F operator is then approximated as 
\begin{eqnarray}
{\cal A}^{{\bf f}+\sigma \bn}_{PF}\approx \frac{\bP_0-I}{\Delta t}=: \bM_0.\label{P_genapprox}
\end{eqnarray}

\begin{remark}\label{remark_edmd}
The EDMD algorithm will arise as a special case of the NSDMD algorithm (\ref{nsdmd_algo})-(\ref{nsdmd2}) where the constraints (\ref{nsdmd2}) in the optimization problem are removed. Similarly, dynamic mode decomposition (DMD) will arise as the special case of the EDMD when the basis function $\bPsi(\bx)$ are chosen to be identity basis functions i.e., $\bPsi(\bx)=\bx$. 
\end{remark}
The problem of convergence of EDMD algorithm for deterministic and stochastic systems is studied in \cite{korda2018convergence,vcrnjaric2020koopman}. The convergence results are asymptotic and guarantee the convergence of the infinite-dimensional Koopman operator in the limit as the data samples, and the number of basis go to infinity. These results will apply to the NSDMD algorithm as well as these operators are naturally positive and preserve Markov property.

For the approximation  of the P-F generator corresponding to control vector field i.e., $\dot \bx=\bg(\bx)$ we use control data set. Let $\{\bx_i,\by_i^\ell,\bu_i\}$ be the control data set obtained by simulating the control system (\ref{cont_syst1}) from different initial condition for $i=1,\ldots, N$ and $\ell=1,\ldots, R$ and $\bu_i$ is assumed to be the step input. In particular, we have 
\begin{equation}
\mathbf{y}_i^\ell=\bx_i + \Delta t{\bf f}(\bx_i) + \sigma\sqrt{\Delta t}\bn(\bx_i)\xi^\ell+\Delta t \bg(\bx_i)
\end{equation}
With step input this data essentially correspond to be generated by stochastic system $\dot \bx={\bf f}(\bx)+\sigma \bn(\bx)\xi+\bg(\bx)$.
Using exactly the same procedure outlined in previous section we construct the approximation of P-F generator  corresponding to the above dynamical system. Following the notation convention from (\ref{P_genapprox}), let 
\begin{eqnarray}
{\cal A}_{PF}^{{\bf f}+\sigma {\bn}+\bg}\approx \bM_1.\label{m1eqn}
\end{eqnarray}
We now use linearity property of the P-F generator to extract the P-F generator corresponding to the dynamical system $\dot \bx=\bg(\bx)$. In particular, we have 
\begin{eqnarray}
{\cal A}_{PF}^{\bg}\rho&=&-\nabla \cdot (\bg \rho)\nonumber\\&=&\left[-\nabla\cdot (({\bf f}+\bg)\rho)+\frac{\sigma^2}{2}\sum_{i,j=1}^n \frac{\partial^2[(\bn \bn^\top)_{ij}\rho]}{\partial \bx_i\partial \bx_j}\right]\nonumber\\&-&\left[-\nabla\cdot ({\bf f}\rho )+\frac{\sigma^2}{2}\sum_{i,j=1}^n \frac{\partial^2[(\bn \bn^\top)_{ij}\rho]}{\partial \bx_i\partial \bx_j}\right]\nonumber\\&=&
\left[{\cal A}_{PF}^{{\bf f}+\sigma {\bn}+\bg}-{\cal A}_{PF}^{{\bf f}+\sigma {\bn}}\right]\rho.
\end{eqnarray}
Hence, the P-F generator for $\dot \bx=\bg(\bx)$ is  approximated as

\begin{eqnarray}
{\cal A}_{PF}^{\bg}\approx \bM_1-\bM_0.\label{approx_generatorG}
\end{eqnarray}
The finite-dimensional approximation of the generator corresponding to the control vector field and uncontrolled vector field, namely $\bM_0$ and $\bM_1$, will be used to approximate the optimization problem. We  discuss this in subsection \ref{section_PFoptimalcontrolfinite}. 

\subsection{Approximation of the Koopman operator}
 For the finite dimensional approximation of the Koopman-based policy interation algorithm in Eqs. (\ref{step1})-(\ref{step2}), we need to approximate the Koopman generator corresponding to the control vector field $\dot \bx=\bg(\bx)$ and uncontrolled vector field $\dot\bx={\bf f}(\bx)+\sigma \bn(\bx)\xi$ using the time-series data generated by $\dot \bx={\bf f}(\bx)+\bg(\bx)u+\sigma \bn(\bx)\xi$. The procedure for this approximation follows exactly along the lines of approximation procedure for the P-F operator as discussed in the previous subsection. The only difference is that we used EDMD algorithm for the approximation (refer to Remark \ref{remark_edmd}). Following (\ref{P_genapprox}) and (\ref{m1eqn}), let $\bK_0,\bL_0$ and $\bK_1,\bL_1$ be the approximation of the Koopman operators and generators corresponding to the uncontrolled, $\dot\bx={\bf f}(\bx)+\sigma \bn(\bx)\xi$ vector field and control vector field with unit step input, $\dot\bx={\bf f}(\bx)+\bg(\bx)\sigma +\bn(\bx)\xi$ respectively i.e., 
\begin{eqnarray}
{\cal A}_K^{{\bf f}+\sigma\bn} \approx \frac{\bK_0-I}{\Delta t}=:\bL_0,\;\;{\cal A}_K^{{\bf f}+\bg +\sigma\bn} \approx \frac{\bK_1-I}{\Delta t}=:\bL_1
\end{eqnarray}

Again using the linearity property of the Koopman generator w.r.t. the vector field, we obtained following approximation for the Koopman generator corresponding the control vector field
\begin{eqnarray}
{\cal A}_K^\bg \approx \bL_1-\bL_0.\label{approximatedg}
\end{eqnarray}
Similarly, time series data from the feedback control system $\dot \bx={\bf f}(\bx)+\bg(\bx)k_k(\bx)+\sigma \bn(\bx)\xi$ is used for the approximation of the Koopman operator and generator at the $k^{th}$ step of the policy iteration. We denote by $\bK_c^k$ and $\bL_c^k$ the Koopman operator and generator for the feedback control system
\begin{eqnarray}
{\cal A}_K^{{\bf f}+\bg(\bx)k_k(\bx)+\sigma \bn(\bx)}\approx \frac{\bK_c^k-I}{\Delta t}=: \bL_c^k\label{finiteKoopmanVI}
\end{eqnarray}

\begin{remark}
In the above construction, we identified the P-F and Koopman generators for the drift vector field and the control vector field using zero and step input. However, it is possible to used arbitrary control inputs for the identification of the two generators. The identification process will correspond to the bilinear lifting of control dynamical system. 
\end{remark}

\subsection{Data-driven Optimal Control: P-F-based Convex Approach}\label{section_PFoptimalcontrolfinite} Let $ \rho(\bx)$, $\bar \rho(\bx)$, and $h_0$ be expressed in terms of the Gaussian RBF as
\begin{eqnarray}
\rho(\bx)\approx \bPsi^\top{\bf v},\;\bar\rho(\bx)\approx \bPsi^\top {\bf w},\;h(\bx)=\bPsi^\top {\bf m}\label{approx}
\end{eqnarray}
With the above representation we have following approximation
\[\nabla\cdot ({\bf f}\rho )-\frac{\sigma^2}{2}\sum_{i,j=1}^n \frac{\partial^2[(\bn \bn^\top)_{ij}\rho]}{\partial \bx_i\partial \bx_j}\approx -\Psi^\top(\bx) \bM_0 {\bf v}\]\[\nabla\cdot(\bg\bar \brho)\approx -\Psi^\top (\bx)(\bM_1-\bM_0){\bf w}\]
We now proceed with the approximation of the cost function. 
\[\int_{\bS}q(\bx)\rho(\bx)d\bx\approx \int_\bX q(\bx)\bPsi^\top d\bx {\bf v}={\bf d}^\top{\bf v}\]
where the vector ${\bf d}:=\int_X q(\bx)\bPsi d\bx$ can be pre-computed. For the approximation of the control cost we make following assumption.
The infinite dimensional optimization problem in terms of the basis function is then written as 

\begin{eqnarray*}
&\min_{\bPsi^\top{\bf v}\geq 0, {\bf w},\kappa>0} {\bf d}^\top {\bf v}+  r  \frac{{\bf w}^\top {\bf D}{\bf w}}{{\kappa}}\\
&{\rm s.t.}\;-\bPsi(\bx)^\top \left(\bM_0 {\bf v}+(\bM_1-\bM_0){\bf w}\right)=\bPsi(\bx)^\top {\bf m}\\
&\bPsi^\top (\bx)\bv\leq \kappa
\end{eqnarray*}
where ${\bf D}:=\int_\bS \bPsi\bPsi^\top d\bx$. Since the basis function $\bPsi$ is known, the matrix $\bf D$ can be evaluated explicitly or computed from data as ${\bf D}\approx \frac{1}{N}\sum_k \bPsi(\bx_k)\bPsi(\bx_k)^\top$ Now using the fact that the basis functions are non-negative and are linearly independent with $s=\|\psi_k\|_\infty$ for $k=1,\ldots, K$, we write the finite dimensional approximation of the optimization problem as 

\begin{eqnarray}
&\min_{{\bf v}\geq 0, {\bf w},\kappa>0} {\bf d}^\top {\bf v}+  \frac{{\bf w}^\top {\bf D}{\bf w}}{\kappa}\nonumber\\
&{\rm s.t.}\; -\left(\bM_0 {\bf v}+ \bM_1{\bf w}\right)= {\bf m},\;\;s \mathds{1}^\top \bv\leq \kappa\label{cvxopt_1}
\end{eqnarray}
where $\mathds{1}$ is a vector of all ones.
The optimal control is then approximated as $u=\frac{\bPsi^\top{\bf w}}{\bPsi^\top{\bf v}}$. 

\subsection{Data-driven Optimal Control: Koopman-based Policy Iteration}\label{Data_KPI}
For the data-driven Koopman-based policy iteration algorithm we proceed as follows. Let 
\[\bPhi(\bx)=[\phi_1(\bx),\ldots,\phi_K(\bx)]^\top\]
be the basis function assumed to be linearly independent. Unlike P-F-based computation, we do not require $\phi_k$ to be non-negative. For the Koopman-based policy iteration we use polynomial basis functions. 
Let, $q(\bx), k_k(\bx)$, and $V_k(\bx)$ be the state cost function, the feedback control input, and value function at the $k^{th}$ step of the value iteration and are approximated as follows:

\begin{equation}
  \begin{aligned}
(q(\bx)+rk_k^2(\bx))\approx \bb^\top\bPhi(\bx),\;\;\;
V_k(\bx)\approx \bW_k^\top \bPhi(\bx). 
  \end{aligned}
 \label{approx_Koopman}
\end{equation}
The approximation of $q+rk_k^2$ is obtained by solving a least square problem. Let $\{\bx_t\}_{t=0}^N$ be the fixed data set uniformly distributed over the state space. Let \[\bar\bPhi=[\bPhi(\bx_0),\ldots, \bPhi(\bx_N)]^\top,\;\; \bar\bq=[q(\bx_0),q(\bx_1),\ldots, q(\bx_N)]^\top\] and $\bar\bk=[k_k^2(\bx_0),\ldots, k_k^2(\bx_N)]^\top$. Following least square problem is solve to determine $\bb$,
\begin{align}
    \min_{\bb}\|\bar \bPhi \bb-(\bar \bq+r\bar \bk)\|
\end{align}
which admits analytical solution as follows.
\[\bb=\bar \bPhi^\dagger (\bar \bq+r\bar \bk)\]
where $\dagger$ denotes pseudo-inverse. 
Using (\ref{finiteKoopmanVI}) and (\ref{approx_Koopman}), we write the finite dimensional approximation of (\ref{step1}) as 

\begin{equation}
     \bPhi(\bx)^\top\bL_c^k \bW_k = -\bPhi(\bx)^\top\bb  .
\end{equation}
Since the $\bPhi$ are basis functions assumed to be independent, (\ref{step1}) can be written as 
\[\bL_c^k \bW_k =-\bb.\]
Using (\ref{finiteKoopmanVI}) the above equation can be written in terms of the Koopman operator as
\begin{align}
   \left( \frac{I-\bK_c^k}{\Delta t}\right)\bW_k=-\bb. \label{W_valuefuention}
\end{align}
There are several ways to solve for the coefficient vector $\bW_k$. In particular, $\bW_k$ can be obtained as 
\[\bW_k=\Delta t(I-\bK_c^k)^{-1}\bb=\sum_{\ell=0}^\infty \Delta t \left[\bK_c^k\right]^\ell \bb\]
where $\left[\bK_c^k\right]^\ell$ stands for the $\ell^{th}$ power of Koopman operator at the $\bK_c^k$ obtained at the $k^{th}$ step of the policy iteration. Alternatively, $\bW_k$ can also be obtained from the least square solution. For the above formula to work we require that spectrum of the Koopman operator is strictly less than one. This is equivalent to stability condition for the feedback control system expressed in terms of the linear Koopman operator  \cite{Vaidya_TAC}. However, from numerical standpoint it is desirable to compute $\bW_k$ in terms of the finite power series of $\bK_c^k$ i.e.,
\begin{align}
\bW_k=\Delta t\left(I+\bK_c^k+\left[\bK_c^k\right]^2+\ldots+\left[\bK_c^k\right]^M\right)\bb
\end{align}
for some finite $M$ which forms an another tuning parameter and could be problem specific. In this paper,  we choose $M=15n$ where $n$ is the dimension of the state space.
Using (\ref{approximatedg}) the approximation to the policy update Eq. (\ref{step2}) for  can be written as

\begin{equation}
k_{k+1}(\bx)=-\frac{r^{-1}}{2}\bPhi(\bx)^\top (\bL_1-\bL_0) \bW_k   .\label{Koopman_policy}
\end{equation}

\begin{remark}
There could be different variants of the above described basic algorithm for the realization of Koopman-based policy iteration algorithm. For example in the above described algorithm we are identifying the Koopman generator for the closed loop system $\dot \bx={\bf f}(\bx)+\bg(\bx)k_\ell(\bx)+\sigma\bn(\bx)\xi$ at the $\ell^{th}$ step of the policy iteration. Instead one could use the linearity of the generator w.r.t. vector field to identify the generator corresponding to only part of vector field that is changing with the iteration i.e., $\bg(\bx)k_\ell(\bx)$.  
\end{remark}




  
  
  

 \subsection{Data-driven Design of Local Stabilizing Control}\label{local_control_data}
 For the  design of local optimal control we employ DMD algorithm for the identification of local system dynamics in the form of $\bA$ and $\bf b$ matrices. In particular, zero and step input data are used for the simultaneous identification of the $\bA$ and ${\bf b}$ matrices as the solution of following optimization problem:
\begin{eqnarray}
\min_{\bA,\bB}\parallel \bY - \bA\bX-{\bf b}{\bf \mathds{1}} \parallel_F\label{dmdalgo}
\end{eqnarray}
where, $\bX=[\bx_0,\bx_1,\ldots, \bx_t,\ldots,\bx_N]$ and $\bY=[\bx_1,\bx_2,\ldots, \bx_{t+1},\ldots,\bx_{N+1}]$ are the collection of time series data generated with step input $u=1$ with $\mathds{1}=[1,\ldots,1]$. These $\bA$ and $\bB$ matrices are employed in the design of local optimal controller in discrete-time setting using MATLAB command $\textit{lqrd}$. The cost function is assumed to be of the form $\bx^\top{\bf Q} \bx+ru^2$, where ${\bf Q}=\frac{\partial^2 q(0)}{\partial \bx^2}$ (assuming the origin to be the equilibrium point). The local optimal controller is active only within the small neighborhood of the equilibrium point.

\subsection{Model-based Approximation of Stochastic Optimal Control}
The model-based approximation of optimal control will be simpler form of the proposed data-driven approximation. In particular, for the model-based approximation we assume that we have access to the system dynamics. Hence, time-series data from the system $\dot \bx={\bf f}(\bx)+\sigma\bn(\bx)\xi$ and control vector field $\dot \bx ={\bf g}(\bx)$ can be used to obtained the approximation of the generators directly instead of using (\ref{approx_generatorG}) and (\ref{approximatedg}) to derive the generators for control vector field.

\subsection{Stochastic Stabilization}
The data-driven stochastic stabilization will arise as the special case of the stochastic optimal control problem. In particular, following \cite{van2006almost}, the condition for feedback stabilization can be written as 
\begin{equation}\label{eq:dualLyapcontrol}
		\nabla \cdot ({\bf f}\rho+ {\bf g} \bar\brho))-\frac{\sigma^2}{2}\sum_{i,j=1}^n \frac{\partial^2[(\bn \bn^\top)_{ij}\rho]}{\partial \bx_i\partial \bx_j}  >0
	\end{equation}
We have introduced new variable $\bar \brho=\bk \rho$. Above inequality is linear in terms of variables $(\rho,\bar \brho)$ and is solved for these variables. The control input $\bk(\bx)$ is then obtained as $\bk(\bx)=\frac{\bar\brho(\bx)}{\rho(\bx)}$. Following discussion from the previous section, we notice that the finite-dimensional approximation of the stabilization constraints can be written as following feasibility problem 
\begin{eqnarray}
-\left(\bM_0 {\bf v}+ \bM_1{\bf w}\right)>0,\;\;\;{\bf v}\geq 0.\label{feasibility}
\end{eqnarray}
where the positivity is component-wise. 
The stabilizing controller is then recovered as $u(\bx)= \frac{\bPsi^\top{\bf w}}{\bPsi^\top{\bf v}}$. 

\section{Simulation results}\label{secction_simulation}

This section presents numerical examples using the P-F-based convex approach and Koopman-based policy iteration method for optimal control. All the simulations codes are developed on  MATLAB  and ran on a computer consisting of 64 GB of RAM and a 3.8 GHz  Intel Core i7  processor. For the finite-dimensional approximation of the P-F operator, we used the Gaussian radial basis function. All the basis functions are assumed to be uniformly distributed in the domain, and the following thump rule is used in selecting $\bar\sigma$ for the Gaussian RBF.  Let $d$ be the distance between the centers of the Gaussian RBF then $\bar\sigma$ is chosen such that $d\leq 3\bar\sigma \leq 1.5d$. We used polynomial basis functions to approximate the Koopman operator in the Koopman-based policy iteration method. For all the example systems, we assumed quadratic cost for the states i.e., $q(\bx)=\bx^\top \bx$ and control i.e., $r u^2$.

\subsection{Stochastic Optimal Control using P-F Operator}
\subsubsection{Scalar Stochastic Example}
\begin{eqnarray}\label{scalar_sys2}
\dot{x}=ax^3 + \sigma x \xi + u
\end{eqnarray}
where $a = 0.01$, $u\in\mathbb{R}$ is an input, and the stochastic noise term has a $\sigma=0.5$. This example is chosen for comparison. For this example, the optimal control can be determined analytically when $\sigma=0$ (i.e., no noise case). We compare the results obtained using our proposed P-F-based approach with the analytically derived control in the deterministic setting. We expect the true stochastic control will be close to the deterministic control, especially for the small value of $\sigma$. The comparison between the analytically derived optimal control for the deterministic system and the control obtained using our proposed P-F-based approach is shown in Fig. \ref{fig:vdp_2}.

For the approximation of stochastic P-F operator, we applied NSDMD algorithm using one-step time-series data with $5$e$^3$ initial conditions and $2$e$^3$ realizations with a sample time of $\Delta t= 0.01$ (i.e., $2^4$ time-series data samples for the two operators). We use $75$ Gaussian RBF as the basis functions $\boldsymbol \Psi(\bx)$, with the radius $\bar\sigma = 0.3443$ and centers distributed uniformly within the range of $[-10,\; 10]$.  

In Fig. \ref{fig:vdp_1},  we show trajectories comparison obtained by applying the analytically derived optimal control with deterministic system dynamics, i.e., $u=-ax^{3}+ x\sqrt{a^2x^4+r^{-1}}$ but applied to the stochastic system (\ref{scalar_sys})  and the data-driven stochastic optimal control obtained using our proposed approach. The trajectories starting from 10 random initial conditions in the domain $[-10,10]$ are shown. The straight magenta  lines define the active region of the local control  parameter $\gamma=2$e$7$ (refer to Eq. \ref{localdensity} for the definition of $\gamma$).   In Fig. \ref{fig:vdp_2}, the analytical and data-driven feedback control $u(x)$ are presented. We can see the close matching between the trajectories as well as the feedback control values.

\begin{figure}[htbp]
\centering
\includegraphics[width=3.0in]{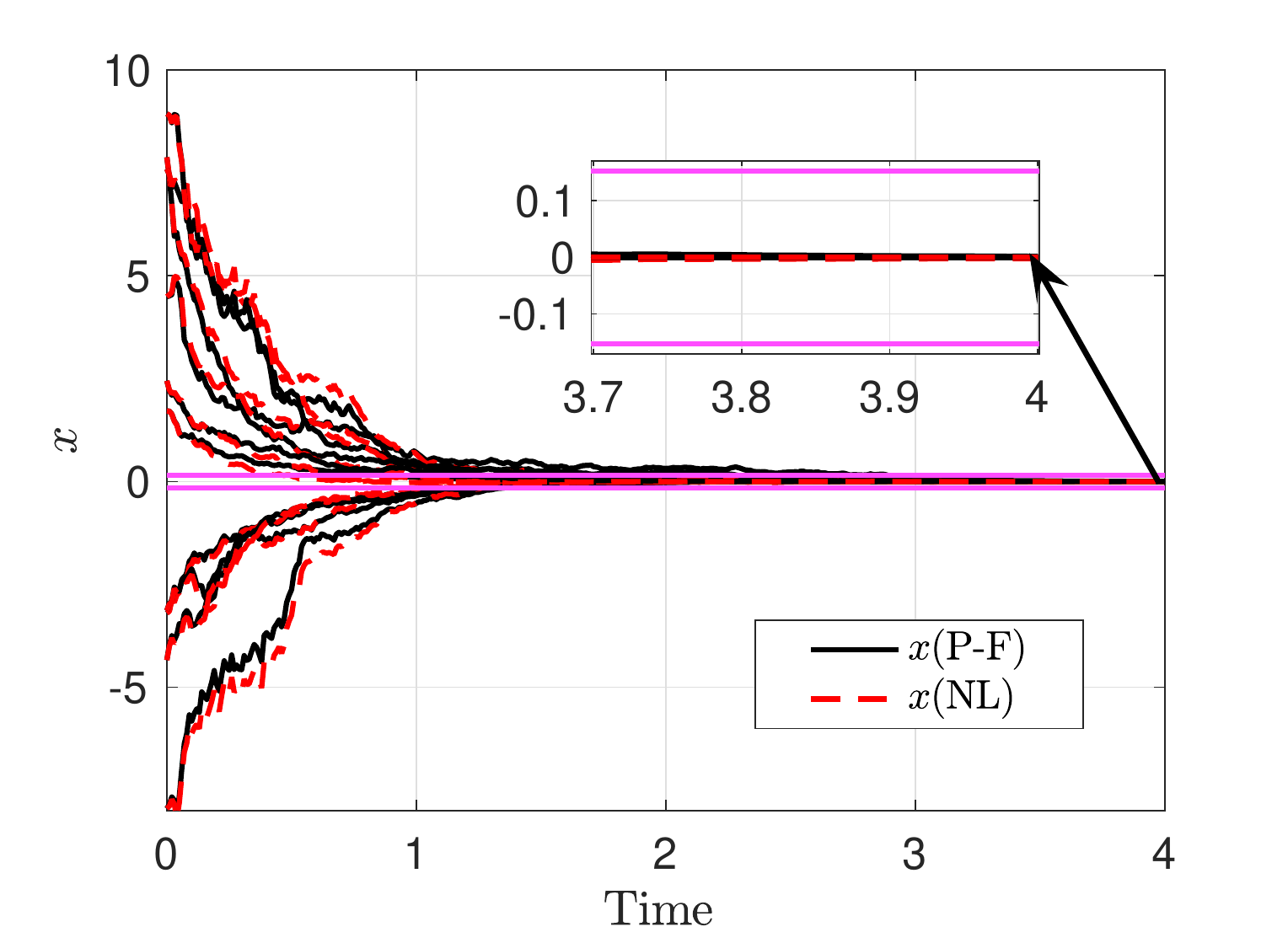}
\caption{Close-loop trajectories comparison using the analytical derived optimal control and P-F-based data-driven stochastic optimal control.}
\label{fig:vdp_1}
\end{figure}
\begin{figure}[htbp]
\centering
\includegraphics[width=3.0in]{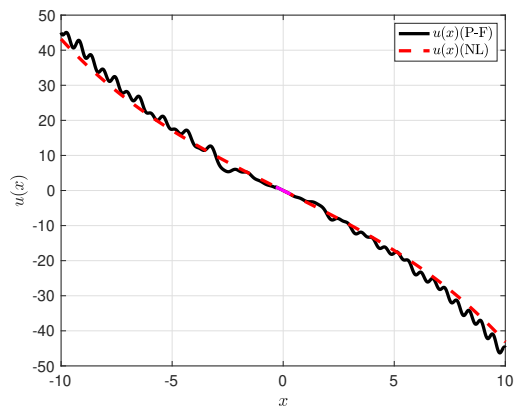}
\caption{Comparison of the feedback control using the analytically derived optimal control and P-F-based data-driven feedback control.}
\label{fig:vdp_2}
\end{figure}

\begin{figure}[htbp]
\centering
\includegraphics[width=3.0in]{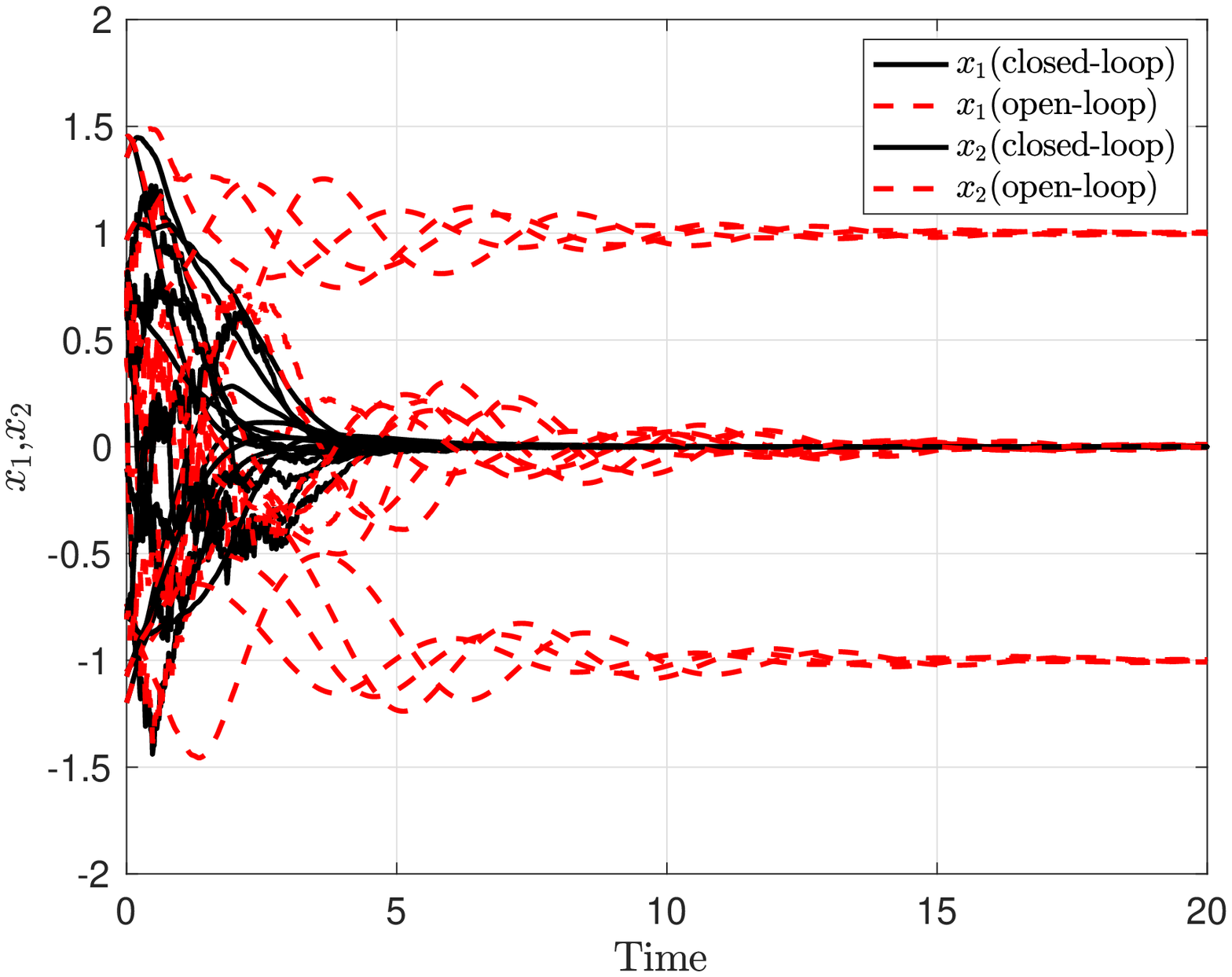}
\hspace{0.1cm}
\includegraphics[width=3.0in]{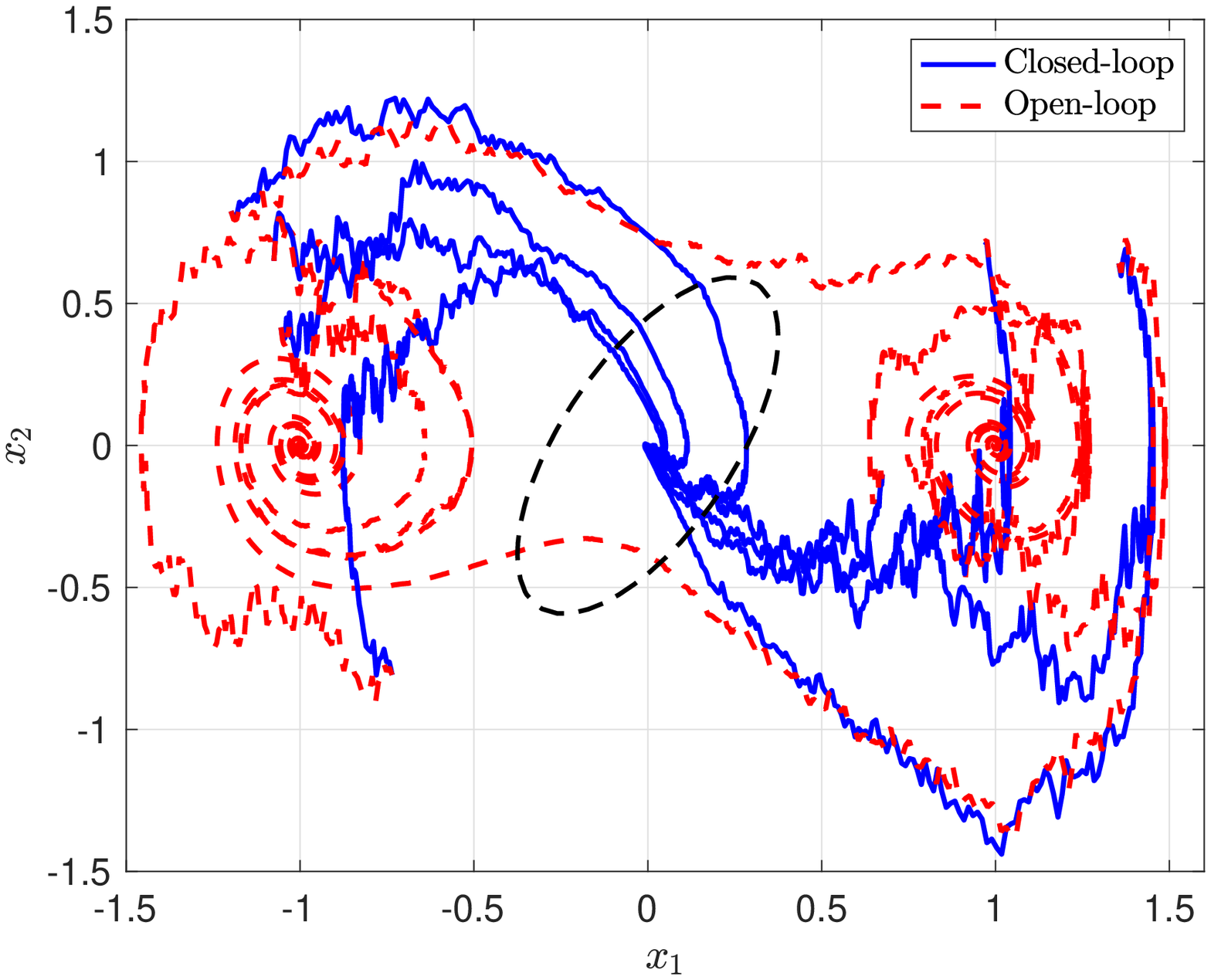}
\caption{Data-driven stochastic optimal control of Duffing oscillator: $x_{1\sim 2}$ vs $time$; Trajectories in 2-D space.}
\label{fig:duffing}
\end{figure}


\subsubsection{Controlled Duffing oscillator}
Our second example is about controlling a duffing oscillator.
\begin{eqnarray}\label{duffing_sys}
\dot{x}_1 = x_2,\;\;\;\;
\dot{x}_2 = x_1 - x_1^3 - 0.5x_2+\sigma x_1\xi +u
\end{eqnarray}
where  $\sigma=0.5$. In this example, we used 225 Gaussian RBF with $\bar{\sigma} = 0.14$ within the range of $D= [-2,\;2]\times [-2,\;2] $ and linear control  parameter $\gamma = 1000$. For the stochastic P-F operator approximation, we applied the NSDMD algorithm using one-step time-series data with $2$e$^5$ initial conditions and $1$e$^3$ realizations with a sample time of $\Delta t= 0.01$.  The region where the blending controller is active is marked by a black doted ellipsoid around the origin in Fig. \ref{fig:duffing}. Simulation results show that the optimal control successfully stabilizes the eight randomly chosen initial conditions to the origin.

\begin{figure}[htbp]
\centering
\includegraphics[width=3.0in]{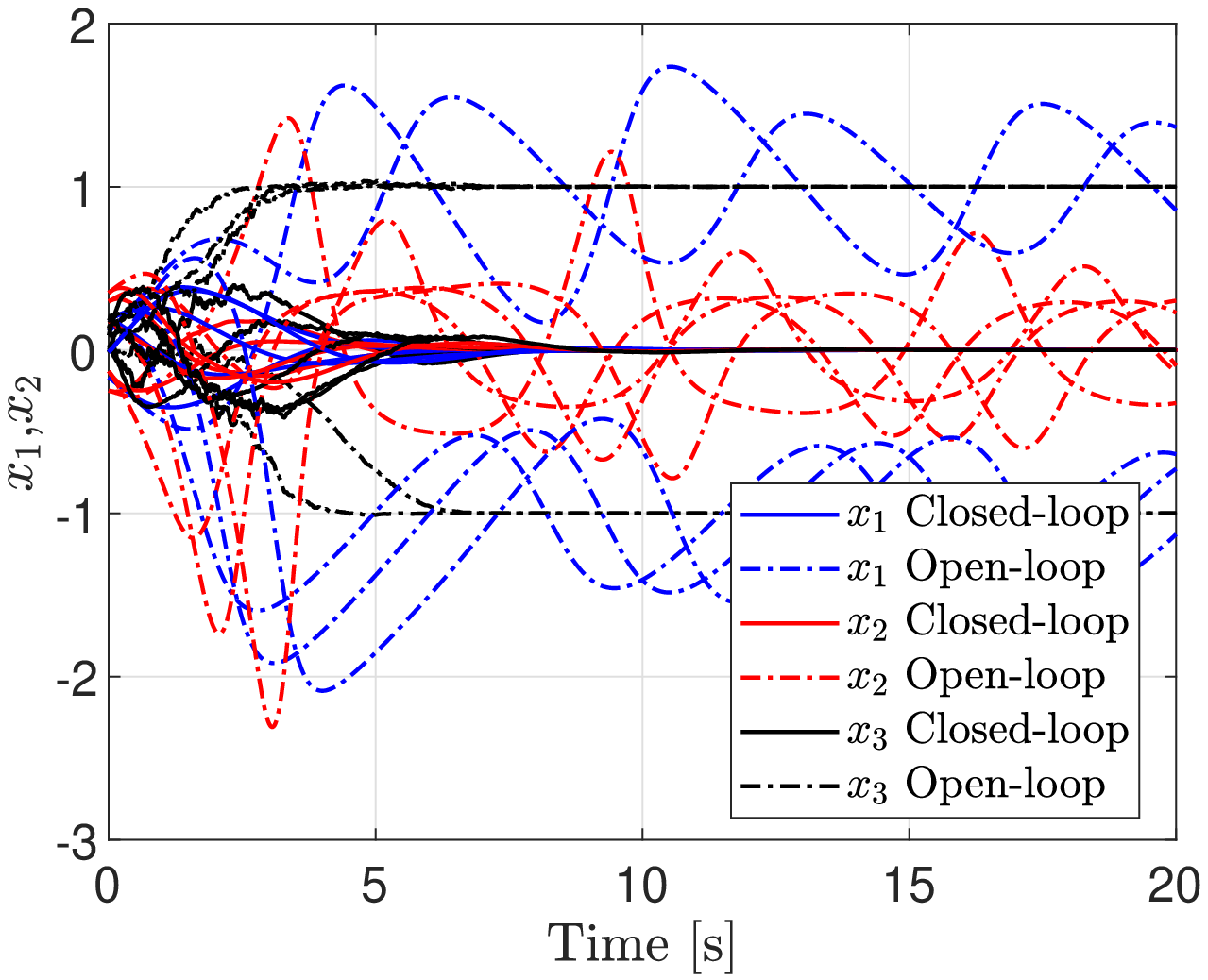}
\hspace{0.1cm}
\includegraphics[width=3.0in]{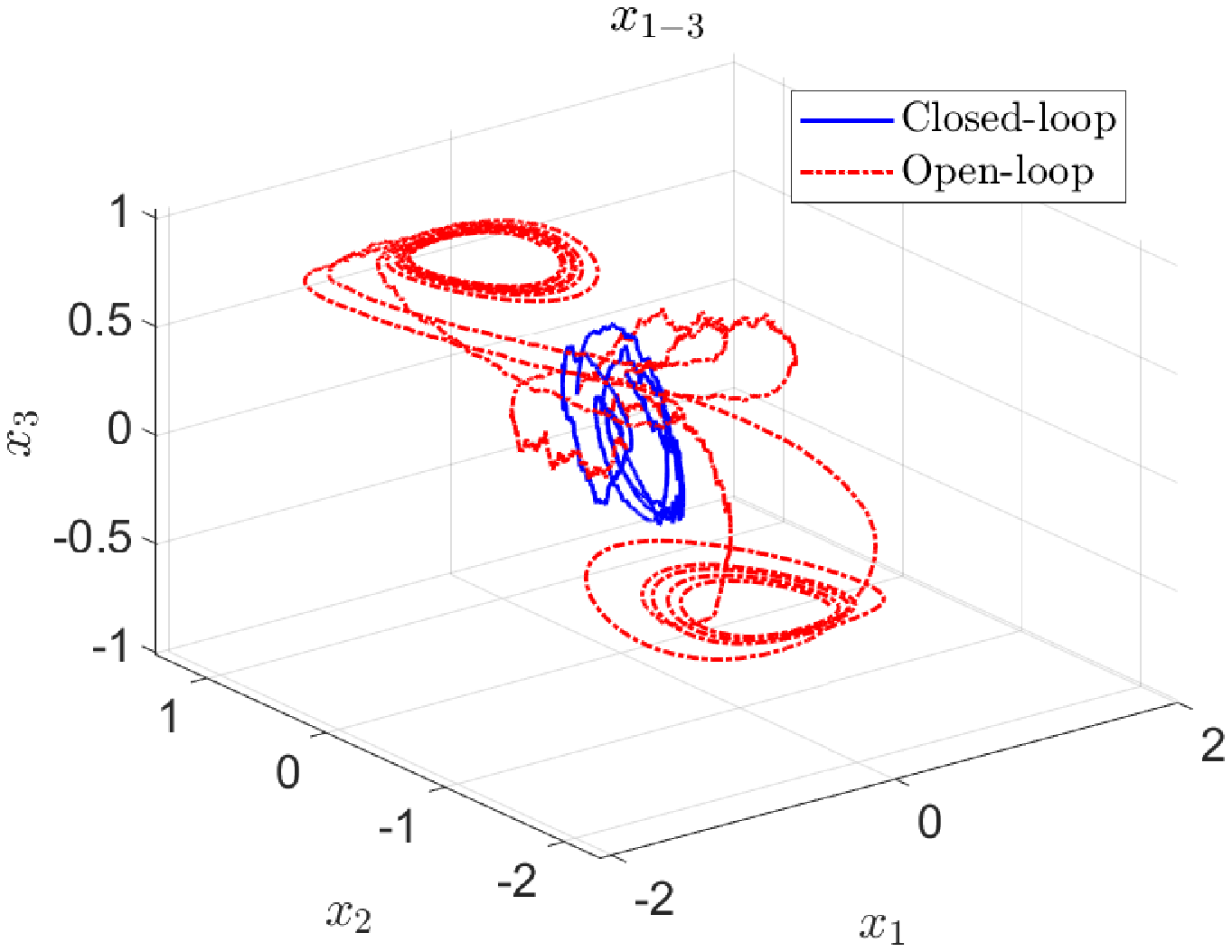}
\caption{Data-driven stochastic  optimal control of 3-D Van der Pol Oscillator: $x_{1\sim 3}$ vs $t$; Trajectories in 3-D space.}
\label{fig:3dVanDerPol}
\end{figure}

\subsubsection{Controlled 3D Van Der Pol Oscillator}
The third example we consider is of 3-D duffing oscillators. 
\begin{equation}
\begin{aligned}
    \dot{x}_1 &= x_2\\
    \dot{x}_2 &=  - x_1 + x_2 -x_3 - x_1^2x_2\\
    \dot{x}_3 &=  x_3 - x_3^2 + \sigma x_1 \xi + 0.5u 
\end{aligned}
\end{equation}

For this example, the stochastic noise term has a $\sigma=0.5$,  we used 512 Gaussian RBF with $\bar{\sigma} = 0.14$ within the range of $D= [-1,\;1]\times [-1,\;1]\times [-1,\;1] $ and linear control  parameter $\gamma = 100$. For the approximation of stochastic P-F operator, we applied NSDMD algorithm using one-step time-series data with $5$e$^5$ initial conditions, and $1$e$^3$ realizations with a sample time of $\Delta t= 0.01$. In Fig. \ref{fig:3dVanDerPol}, we show the results for the optimal control of 3D oscillator system with five random initial conditions.

\begin{figure}[htbp]
\centering
\includegraphics[width=3.0in]{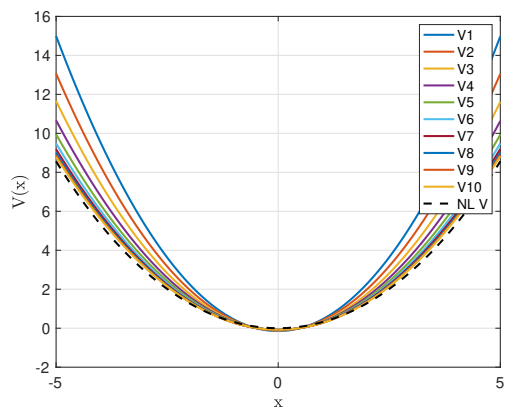}
\caption{Data-driven approximation of value functions using Koopman operator.}
\label{fig:VFKPI}
\end{figure}
%

\begin{figure}[htbp]
\centering
\includegraphics[width=3.0in]{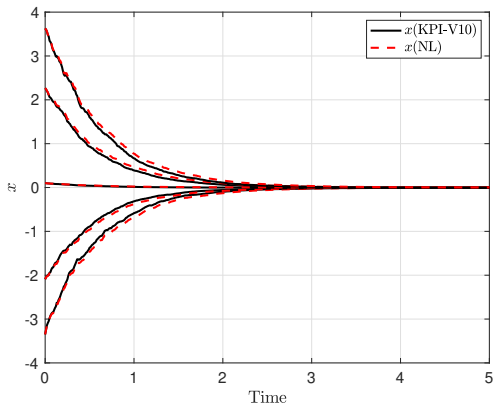}
\caption{Close-loop  trajectories  for  analytical  and  data-driven stochastic optimal control using the last value function (i.e., V10).}
\label{fig:VFKPIwithNL}
\end{figure}
%


\subsection{Stochastic Koopman-based Policy Iteration}

In this subsection, we present the results of optimal control using the Koopman operator.

\subsubsection{Scalar stochastic Example}
Consider  the scalar example again

\vspace{-0.1in}
\begin{eqnarray}\label{scalar_sys}
\dot{x}=ax^3 + \sigma x \xi + u
\end{eqnarray}
where $a = 0.01$, $u\in\mathbb{R}$ is an input, and the stochastic noise term has a $\sigma=0.1$. For each iteration, we approximate the close-loop Koopman generator using a polynomial basis of the first four monomials  $\{ x^n\}_0^3$. We applied EDMD algorithm (refer to Remark \ref{remark_edmd}) with one-step time-series data and $5$e$^2$ initial conditions. $1$e$^2$ different stochastic realizations were used with a sample time of $\Delta t= 0.01$ within the range of $D= [-5,\;5]$. The initial policy $k_0$ is an LQR gain with parameters $Q=1$ and $R=0.1$. Fig. \ref{fig:VFKPI} presents the value function approximations for the first ten iterations. In Fig. \ref{fig:VFKPI}, we show the plot of the cost value as the function of iteration. We notice that the cost function is decreasing the increase in the policy iteration. The dotted line denotes the analytically derived optimal value function and is given by


\begin{eqnarray*}
V^\star_{det}=&\frac{1}{2}a R x^4-Rx^2\sqrt{a^2 x^4+R^{-1}}(a^2 R x^4 + 1)\\
&\left( \frac{1}{2(a^2Rx^4+1)}+ \frac{\sinh^{-1}(a\sqrt{R}x^2)}{2 a \sqrt{R} x^2(a^2 R x^4 +1)^{3/2}} \right)
\end{eqnarray*}
We notice that the cost function for the stochastic system is approaching $V^\star_{det}$.



\begin{figure}[htbp]
\centering
\includegraphics[width=3.0in]{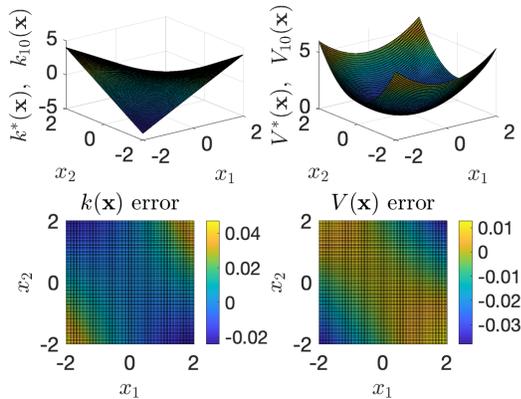}
\caption{Comparison between optimal value function and optimal control using analytical derivation and data-driven optimal control Koopman-based policy iteration.}
\label{fig:comparVandU}
\end{figure}
\begin{figure}[htbp]
\centering
\includegraphics[width=3.0in]{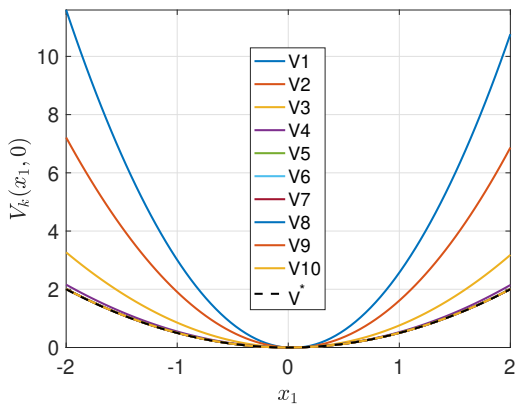}
\caption{Data-Driven value function with fixed $x_2=0$.}
\label{fig:sliceVF}
\end{figure}
\begin{figure}[htbp]
\centering
\includegraphics[width=3.0in]{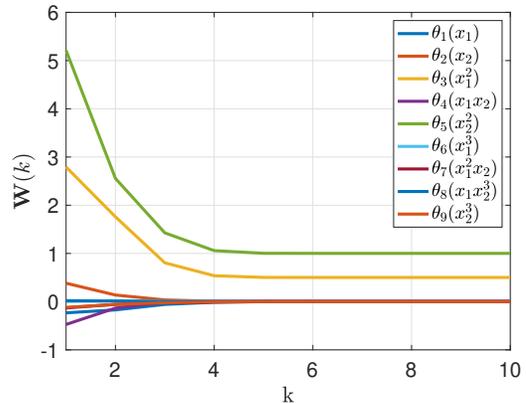}
\caption{Parameters convergence of optimal value function.}
\label{fig:paramatersVF}
\end{figure}
\begin{figure}[htbp]
\centering
\includegraphics[width=3.0in]{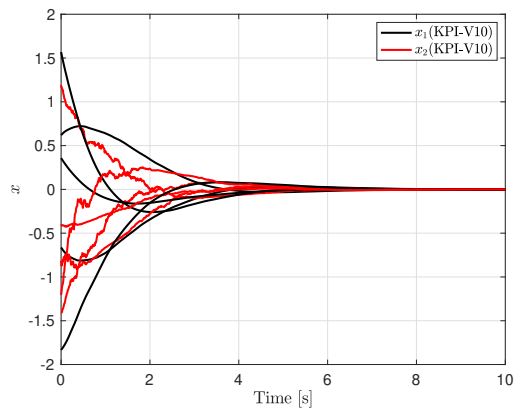}
\caption{Close-loop  trajectories  for   data-driven Koopman-based policy iteration  using  the  last  value  function  (i.e.,V10).}
\label{fig:trajVF}
\end{figure}

\subsubsection{Stable Stochastic System}
Consider the following two-dimensional system, where the uncontrolled system has a stable equilibrium point.
\begin{align}\label{stabel_sys}
&\dot{x}_1 = -x_1 + x_2,\nonumber\\
&\dot{x}_2 = -0.5(x_1 + x_2) + 0.5 x_1^2 x_2+\sigma x_1\xi +u
\end{align}
where $\sigma =0.1$. For this example the analytical form of optimal control can be written explicitly for the case $\sigma=0$ and is of the form
\[u^\star=-x_1x_2,\;\;\;V^\star(\bx)=0.5 x_1^2+x_2^2.\]
For each step of policy iteration, we approximate the Koopman generator for the closed-loop system using a monomial basis of order 1 to 4 $\{ \bx^n\}_1^4$ (i.e., nine elements). We applied EDMD algorithm (refer to Remark \ref{remark_edmd}) with one-step time-series data and $3$e$^4$ initial conditions. $1$e$^2$ different stochastic realizations were used with a sample time of $\Delta t= 0.01$ in the training domain $D= [-2,\;2]\times [-2,\;2]$. Since the system has a stable equilibrium point, we choose the initial control $k_0=0$. For illustration purposes, we present the following comparison with a deterministic case for a small value of $\sigma$. The left plots in Fig. \ref{fig:comparVandU} present  data-driven KPI and the optimal control derived analytically for the deterministic case in \cite{luo2014data}. At the same time, the right plots show the value function comparison with its respective error. Both data-driven optimal control and value function are approximated with the last iteration (i.e., $V_{10}(\bx)$ and $k_{10}(\bx)$). In Fig. \ref{fig:sliceVF}, we show the plot of the cost value as the function of iteration with fixed $x_2=0$ . We notice that the cost function is decreasing with the increase in the policy iteration. The dotted line denotes the analytically derived optimal value function. The parameters of value function evolving in each iteration are shown in Fig.  \ref{fig:paramatersVF}. Note that the convergence of parameters goes to optimal values, which are $\theta_3=0.5$, $\theta_5=1$, and the equal to zero. Finally, the close-loop trajectories for ten random initial conditions using data-driven  Koopman-based policy iteration with $V_{10}$ are presented in Fig. \ref{fig:trajVF}.

\section{Conclusion}\label{section_conclusion}
The duality in the SOCP is discovered through the lenses of linear transfer operator theory involving Koopman and P-F operators. Our first main result provides a convex formulation of the SOCP using the P-F-based lifting of nonlinear system dynamics. Our second main result establishes a connection between the Koopman operator and the HJB equation. This connection allows us to develop a new algorithm for the data-driven approximation of stochastic optimal control, called Koopman-based policy iteration (KPI). 
A computation framework based on the finite-dimensional approximation of the P-F and Koopman operators is developed to compute stochastic optimal control. Simulation results are presented to demonstrate the application of the developed theoretical and computational framework. 


\section{Appendix}\label{section_appendix}

\begin{proof}
For any set $B\in{\cal B}(\bS)$, we have
\[ \bX_t^\bx\in B,\;\;{\rm iff} \;\;\bx\in B_t:=\{\bx\in \mR^n: \bX_t^\bx\in B\} \]
and hence
$\lambda_{B}(\bX_t^\bx)= \lambda_{B_t}(\bx)$. We have
\begin{eqnarray*}{\rm Prob}(\bX_t^\bx\in B)=\mE_\bx[\chi_{B_t}(\bx)]=\mE_\bx[\chi_B(\bX_t^\bx)]=[\mU_t^c\chi_B](\bx).\end{eqnarray*}

Since the origin is assumed to be a.e. a.s. stable w.r.t. measure $\mu$, we have for $\mu$ almost all initial condition $\bx$  \begin{eqnarray}
 {\rm Prob}\{\lim_{t\to \infty} \bX_t^\bx= 0\}=1\implies \nonumber\\{\rm Prob}\{\lim_{t\to \infty}\bX_t^\bx\in B\}=0, \nonumber
\end{eqnarray}
for all set $B\in {\cal B} (\bS)$. Furthermore, since $B\in {\cal B} (\bS)$ and the origin is in the interior of $\bS^c$, we have using dominated convergence theorem  as $\lambda_{B_t}(\bx)\leq 1$ for all $t$,
\begin{eqnarray}
0={\rm Prob}\{\lim_{t\to \infty} \bX_t^\bx\in B\}=\mE_\bx[\lim_{t\to \infty} \chi_{B_t}(\bx)]\nonumber\\=\lim_{t\to \infty}\mE_\bx[\chi_B(\bX_t^\bx)]=\lim_{t\to \infty} [\mU_t^c \chi_B](\bx). \nonumber
\end{eqnarray}
for almost all initial condition $\bx$ w.r.t measure $\mu$. Hence using the fact that the $\mu$ has density $h$, we obtain

 \begin{eqnarray*}0=\int_{\mR^n}\lim_{t\to \infty} [\mathbb{U}_t^c\chi_B](\bx)h(\bx)d\bx\nonumber\\=\int_{\mR^n}\chi_B(\bx)\lim_{t\to \infty}[\mathbb{P}_t^c h](\bx)d\bx.
 \end{eqnarray*}
 where we have used the duality between the Koopman and P-F semi-groups (\ref{duality_semigroup}).  
  The above is true for arbitrary set $B\in{\cal B}( \bS)$, hence we have $\lim_{t\to \infty}[\mathbb{P}_t^c h](\bx)=0 $.
  For any given $\epsilon>0$ and $T\in \mathbb{Z}^+$, let
  \[R_T=\{\bx\in \bX : {\rm Prob}(\bX_t^\bx\in \bS)\geq \epsilon\;{\rm for\;some}\;t>T\}\] and
  \[R=\bigcap\limits_{T=1}^\infty R_T\]
  So set $R$ consists of points with probability larger than $\epsilon$ to end up in set $\bS$. From the construction of set $R$, we have
  
  \[\bx\in R \iff {\rm Prob}\{\bX_t^\bx\in R\}=1,\;\;\;\;\forall t\geq 0.\]
  In other words, $\lambda_R(\bx)=1$ if and only if ${\rm Prob}\{\bX_t^\bx\in R\}=1$. To prove the results it is sufficient to show that $\mu_0(R)=0$. This is true because we know that the origin is locally almost sure stable (Remark \ref{remark_localstability}). For all $t\geq 0$, we have 
 \[\mu_0(R)=\int_{\mR^n}\lambda_R(\bx)d\mu_0(\bx)=\int_{\mR^n} {\rm Prob}\{\bX_t^\bx\in R\}d\mu_0(\bx)\] 
 Now,
 \[\int_{\mR^n} {\rm Prob}\{\bX_t^\bx\in R\}d\mu_0(\bx)=\int_{\mR^n}[\mU_t \chi_D](\bx)h_0(\bx)d\bx\]
 and using duality we have 
 \[\mu_0(R)=\int_{\mR^n}[\mU_t \chi_R](\bx)h_0(\bx)d\bx=\int_{\mR^n}[\mP_t h_0](\bx)d\bx\]
 Since the above is true for all $t\geq 0$, using dominated convergence theorem, we have
 \[\mu_0(R)=\lim_{t\to\infty}\int_{\mR^n}[\mP_t h_0](\bx)d\bx=\int_{\mR^n}\lim_{t\to\infty}[\mP_t h_0](\bx)d\bx=0.\]
 
\end{proof}

\qed

\bibliographystyle{IEEEtran}
\bibliography{root.bib}
\end{document}